\documentclass[11pt, a4paper]{article}
\usepackage{mathrsfs}
\usepackage{graphicx}
\usepackage{lscape}
\usepackage{enumerate}

\usepackage{amsmath,amsthm}
\usepackage{amssymb}
\usepackage{amsbsy}

\usepackage{amsfonts}
\usepackage{longtable}
\usepackage{rotating}
\usepackage{multirow}

\newtheorem{lem}{Lemma}[section]%
\newtheorem{theorem}[lem]{Theorem}%
\newtheorem{defi}[lem]{Definition}%
\newtheorem{prop}[lem]{Proposition}%
\newtheorem*{note}{Note}%

\renewcommand\a{\alpha} \renewcommand\b{\beta} \newcommand\g{\gamma} \renewcommand\d{\delta}

\newcommand\G{\Gamma}

\renewcommand\o{\omega}

\newcommand\nd{\mathrel{\bigm|\kern-.7em/}}

\newcommand\f{\noindent}

\newcommand\PSL{\mathrm{ PSL}}
\newcommand\AGL{\mathrm{ AGL}}
\newcommand\PL{\mathrm{ P}\Gamma\mathrm{L}}
\newcommand\AGammaL{\mathrm{A}\Gamma\mathrm{L}}
\newcommand\GammaL{\Gamma\mathrm{L}}
\newcommand\POmega{\mathrm{ P}\Omega}

\newcommand\SL{\mathrm{ SL}}
\newcommand\GL{\mathrm{GL}}
\newcommand\GF{\mathrm{ GF}}

\newcommand\GU{\mathrm{GU}}

\newcommand\GO{\mathrm{GO}}
\newcommand\Sp{\mathrm{ Sp}}

\newcommand\rank{\mathrm{rank}}

\newcommand\Aut{\mathrm{Aut}}

\newcommand\soc{\mathrm{ soc}}

\newcommand\Cay{\mathrm{ Cay}}

\newcommand\Sym{\mathrm{ Sym}}
\newcommand\Alt{\mathrm{ Alt}}

\newcommand\mz{{\mathbb Z}}

\newcommand\D{\Delta}
\newcommand\A{{\rm A}}

\newcommand\GG{\mathcal{G}}
\newcommand\K{{K}}

\newcommand\Ome{{\rm\Omega}}

\newcommand\norml{\trianglelefteq}

\usepackage{color}

\definecolor{Thistle}{rgb}{0.847,0.749,0.847}
\definecolor{Khaki}{rgb}{0.941,0.902,0.549}
\definecolor{Orchid}{rgb}{0.855,0.439,0.839}
\definecolor{MediumOrchid}{rgb}{0.729,0.333,0.827}

\definecolor{brown}{rgb}{0.8,0.5,0}
\definecolor{LightBrown}{rgb}{0.8,0.2,0.4}

\definecolor{DarkGray}{rgb}{0.78,0.78,0.78}
\definecolor{DarkMidGray}{rgb}{0.81,0.81,0.81}
\definecolor{MidGray}{rgb}{0.85,0.85,0.85}
\definecolor{LightGray}{rgb}{0.88,0.88,0.88}
\definecolor{VeryLightGray}{rgb}{0.96,0.96,0.96}

\definecolor{GrayA}{rgb}{0.7,0.7,0.7}
\definecolor{GrayB}{rgb}{0.78,0.78,0.78}
\definecolor{GrayC}{rgb}{0.80,0.80,0.80}
\definecolor{GrayD}{rgb}{0.82,0.82,0.82}
\definecolor{GrayE}{rgb}{0.84,0.84,0.84}
\definecolor{GrayF}{rgb}{0.86,0.86,0.86}
\definecolor{GrayG}{rgb}{0.88,0.88,0.88}
\definecolor{GrayH}{rgb}{0.90,0.90,0.90}
\definecolor{GrayI}{rgb}{0.92,0.92,0.92}
\definecolor{GrayJ}{rgb}{0.94,0.94,0.94}

\definecolor{VeryLightBlue}{rgb}{0.9,0.9,1}
\definecolor{LightBlue}{rgb}{0.8,0.8,1}
\definecolor{MidBlue}{rgb}{0.5,0.5,1}
\definecolor{DarkBlue}{rgb}{0,0,0.6}

\definecolor{Gold}{rgb}{1,0.843,0}

\definecolor{LightGreen}{rgb}{0.88,1,0.88}
\definecolor{MidGreen}{rgb}{0.6,1,0.6}
\definecolor{DarkGreen}{rgb}{0,0.6,0}

\definecolor{VeryLightYellow}{rgb}{1,1,0.9}
\definecolor{LightYellow}{rgb}{1,1,0.6}
\definecolor{MidYellow}{rgb}{1,1,0.5}
\definecolor{DarkYellow}{rgb}{1,1,0.2}

\definecolor{VeryLightRed}{rgb}{1,0.9,0.9}
\definecolor{LightRed}{rgb}{1,0.8,0.8}
\definecolor{MidRed}{rgb}{1,0.55,0.55}

\begin{document}

\title{On primitive $2$-closed permutation groups of rank at most four}
\vspace{4 true mm}

\author {{\sc Michael Giudici}\\
{\small\em Department of Mathematics and Statistics}\\
{\small\em The University of Western Australia}\\
{\small\em 35 Stirling Highway, Perth, WA 6009, Australia}\\
{ \texttt{michael.giudici@uwa.edu.au}}\\[1ex]
{\sc Luke Morgan}\\
{\small\em UP FAMNIT}\\
{\small\em University of Primorska}\\
{\small\em Glagolja\v{s}ka 8, Koper, 6000, Slovenia}\\
{ \texttt{luke.morgan@famnit.upr.si}}\\[1ex]
{\sc Jin-Xin Zhou}\\
{\small\em Department of Mathematics}\\
{\small\em Beijing Jiaotong University}\\
{\small\em Beijing 100044, P.R. China}\\
{ \texttt{jxzhou@bjtu.edu.cn}}\\[1ex]
}

\date{}
\openup 0.5\jot

\maketitle

\begin{abstract}
We characterise the primitive 2-closed groups $G$ of rank at most four that are not the automorphism group of a graph or digraph and show that if the degree is at least 2402 then there are just two infinite families or $G\leqslant \AGammaL_1(p^d)$, the 1-dimensional affine semilinear group. These are the first known examples of non-regular 2-closed groups that are not  the automorphism group of a graph or digraph.

\medskip

\end{abstract}

\thispagestyle{empty}

\section{Introduction}
Let $G$ be a transitive permutation group on a set $\Ome$. We say that $G$ is {\em regular} on $\Omega$ if $G_\a=1$ for every $\a\in \Omega$, where $G_\a$ is the \emph{stabiliser} of $\a$ in $G$, defined as the subgroup of $G$ consisting of those elements of $G$ that fix $\a$.
The {\em $2$-closure $G^{(2)}$} of $G$ is the largest subgroup of $\Sym(\Ome)$  containing $G$ which has the same orbits as $G$ in the induced action on $\Ome\times\Ome$. Clearly $G \leqslant G^{(2)}$ and if $G=G^{(2)}$, then $G$ is said to be {\em 2-closed}. Two important classes of $2$-closed groups are regular groups and the automorphism groups of graphs or digraphs. A $2$-closed group may not be the automorphism group of a graph or digraph.  For example, $V_4$ is  regular on the set $\{1,2,3,4\}$, and is therefore $2$-closed, but it is not the automorphism group of any graph or digraph of order four.

 If $\Gamma$ is a  graph or digraph on $\Omega$ and $G\leqslant \Aut(\Gamma)$ then $\Gamma$ is a union of orbital digraphs for $G$ (see Section  \ref{prelim}) and $G^{(2)}\leqslant\Aut(\Gamma)$. Hence $G\leqslant G^{(2)}\leqslant \Aut(\Gamma)$.  Thus if $G\neq G^{(2)}$ then $G$ is not the automorphism group of any graph or digraph. In this paper we seek to determine when there is a graph or digraph $\Gamma$ such that  $G=G^{(2)}=\Aut(\Gamma)$.

The question of when a regular group is the automorphism group of a graph or digraph has already been solved. A graph (or digraph) $\G$ is called a {\em GRR} (or {\em DRR}) of a finite group $G$ if the automorphism group $\Aut(\G)$ of $\G$ is a regular permutation group on $V(\G)$ isomorphic to $G$. The DRR-problem (finding all of the finite groups which admit DRRs) was solved by Babai in \cite{Babai1980} where he proved that, with five exceptions, every finite group admits a DRR. Finding those finite groups admitting GRRs was also an active topic, and the GRR-problem was completely solved in the 1980s (see, for example, \cite{Godsil1981,Godsil1978,Godsil1983}).

Some progress has been made in the study of non-regular $2$-closed permutation groups that occur as the automorphism group of some graph or digraph. See for example, \cite{Xuj-JGroupT2010,Xuj-dm2011,Xuj-e-jc2015}. In particular, Jing Xu \cite{Xuj-e-jc2015} proved that every $2$-closed permutation group that contains a cyclic regular subgroup is the automorphism group of a digraph.
However, as pointed out in \cite{XUmy2008-JAMS} by Ming-Yao Xu, no examples are known of non-regular $2$-closed groups that are not the automorphism groups of graphs or digraphs.
Let $N_2 R$ be the set of numbers $n$ for which there exists a $2$-closed transitive group of degree $n$ without a regular subgroup, and let $NC$ be the set of numbers $n$ for which there exists a vertex-transitive graph of order $n$ that is non-Cayley.
 Ming-Yao Xu  \cite{XUmy2008-JAMS} posed the question of determining $N_2R\setminus NC$. He said that to study this question, ``we should first find non-regular $2$-closed groups that are not the automorphism groups of (di)graphs. We do not know of any such examples." Based on this, in 2012, Jing Xu~\cite{Xuj-2012} posed the following problem:

\medskip
\f{\bf Problem~A}\ Find an infinite family of $2$-closed primitive groups that are not the automorphism group of any graphs or digraphs.
\medskip

The \emph{rank} of a permutation group $G$ on a set $\Omega$ is the number of orbits of $G$ on $\Omega\times\Omega$. Note that the lowest possible rank of a non-trivial permutation group is two -- one orbit being the diagonal  $\{(\alpha,\alpha) \mid \alpha \in \Omega\}$. 
It is not difficult to see that every $2$-closed permutation group of rank two or three is the automorphism group of one of its orbital graphs (see Lemma~\ref{primitive}).
So a $2$-closed primitive group that is not the automorphism group of any graph or digraph has rank at least $4$. In the literature, there are some impressive results on the primitive permutation groups of low rank.
For example, the classification of finite primitive groups of rank at most $5$ has been completed except for the subgroups of affine groups (see \cite{Bannai,Cuypers,Praeger-Soicher}), and the classification of finite affine primitive permutation groups of rank at most $3$ has also been completed (see \cite{Huppert-2-trans,Liebeck-affine}). More recently,  Muzychuk and Spiga~\cite{Muzy-Spiga-JACO2020} classified the finite almost simple primitive groups of rank at most $8$ with
socle a sporadic simple group and the finite almost simple primitive groups of rank at most $n^2$ with
socle an alternating group $\A_n$.

Motivated by the facts listed above,  we aim to classify the primitive $2$-closed permutation groups of degree $n$ and rank at most $4$ that are not the automorphism group of any graph or digraph. In this paper, we shall first classify these groups which are not subgroups of $\AGammaL_1(p^d)$, where $n=p^d$ is not a prime power. Using \cite{Colva4096,Dixon1000,RD-Uaffine1000}, a complete classification of primitive permutation groups of degree less than $4096$ is available to us, and so we use  computer algebra packages such as {\sc Magma} \cite{BCP} to aid our investigation for groups of such degrees.

The following is the main result of this paper. 

\begin{theorem}\label{main-theorem}
Let $G$ be a primitive $2$-closed permutation group of degree $n$ and rank at most $4$ such that $G\not\leqslant\AGammaL_1(p^d)$ with $n=p^d$. Then there is a graph or digraph $\Gamma$ such that $G=\Aut(\Gamma)$ if and only if  $G$ is not permutationally isomorphic to one of the groups below:
\begin{enumerate}[$(1)$]
  \item    $\mathcal G(m)=V\rtimes (D_8\circ \GL_m(3)) \leqslant \AGL_{2m}(3)$, for  $m\geqslant 2$, where $\mathcal G (m)$ has regular normal subgroup $V$ and  $D_8 \circ \GL_m(3)$ preserves the tensor product decomposition $V=X \otimes Y$, where $X$ is a $2$-dimensional $\GF(3)$-space on which $D_8$ preserves a direct sum decomposition and $ Y$ is an $m$-dimensional $\GF(3)$-space on which $\GL_m(3)$ acts naturally.
  \item $\mathcal H(m)=V\rtimes ((C_3\wr C_2)\circ \GL_m(4)).2 \leqslant \AGL_{4m}(2)$, for  $m\geqslant 2$, where $\mathcal H (m)$ has regular normal subgroup $V$ and  $((C_3\wr C_2) \circ \GL_m(4)).2\leqslant\GammaL_{2m}(4)$ preserves the tensor product decomposition $V=X \otimes Y$, where $X$ is a $2$-dimensional $\GF(4)$-space on which $C_3\wr C_2$ preserves a direct sum decomposition and $ Y$ is an $m$-dimensional $\GF(4)$-space on which $\GL_m(4)$ acts naturally.
\item $G$ is $\mathrm{PrimitiveGroup}(n,k)$ of \textsc{Magma}'s database of primitive groups where $(n,k)$ is one of $(25,11)$, 
$(64,27)$, 
$(81,77)$, $(81,87)$,
$(169,41)$,  $(625,547)$, 
and $(2401, 991)$.
\end{enumerate}
\end{theorem}

We note that examples do arise of groups $G\leqslant\AGammaL_1(p^d)$ that are not the automorphism group of a graph or digraph. We give an infinite family in Section \ref{sec:eg1diml} and further small examples in Section \ref{sec:smalleg}, but do not attempt to classify them. Thus, we provide three infinite families of primitive $2$-closed permutation groups that are not the automorphism group of any graph or digraph, and so a solution of Problem~A is given.

\section{Preliminaries}\label{prelim}

All groups in this paper are finite. All group actions and graph isomorphisms are written on the right, and basic group theoretic terminology may be found in \cite{WI}.

\subsection{Notation and definitions}
For two positive integers $a,b$, we denote by $\gcd(a,b)$ the greatest common divisor of $a$ and $b$.
For a positive integer $n$, we denote  by $C_{n}$ the cyclic group of order $n$,  by $\mathbb{Z}_{n}$ the ring of integers modulo $n$, by $\mathbb{Z}_{n}^{\ast}$ the multiplicative group of $\mathbb{Z}_{n}$ consisting of numbers coprime to $n$, by $D_{2n}$ the dihedral group of order $2n$, by $A_n$ the alternating group of degree $n$ and by $S_n$ the symmetric group of degree $n$. For two groups $M$ and $N$, $N\rtimes M$ denotes a semidirect product of $N$ by $M$. Given a group $G$, denote by $1$, $\Aut(G)$, $Z(G)$ and $\soc(G)$ the identity element, automorphism group, center and socle of $G$, respectively. For a subgroup $H$ of $G$, denote by $C_G(H), N_G(H)$ the centraliser and normaliser of $H$ in $G$, respectively. Of course $C_G(H)$ is normal in $N_G(H)$, and the quotient group $N_G(H)/C_G(H)$ is isomorphic to a subgroup of $\Aut(H)$.

For a graph or digraph $\G$, we denote by $\Aut(\G)$ the automorphism group of $\G$. When $\Aut(\G)$ acts transitively on ordered pairs of vertices of a graph at distance $i$ for each integer $i\geqslant 0$, we say that $\G$ is {\em distance-transitive}. 

Let $\G$ be a finite digraph. A sequence $v_0,v_1,\dots,v_m$ of vertices of $\G$ is called an {\em undirected path} of length $m$ from $v_0$ to $v_m$ if for each $i$, there is an edge in $\G$ from $v_i$ to $v_{i+1}$ or an edge from $v_{i+1}$ to $v_i$. We say that $\G$ is {\em connected} if for every pair of vertices $u$ and $v$ there is an undirected path from $u$ to $v$.

Let $\D=\{0,1,\ldots,k-1\}$ and let $d\geqslant 2$ be an integer. The {\em Hamming graph} $H(d,k)$ has vertex set $\D^d$, the set of ordered $d$-tuples of elements of $\D$, or sequences of length $d$ from $\D$. Two vertices are adjacent if they differ in precisely one coordinate. The Hamming graph has valency $d(k-1)$ and diameter $d$. It is distance-transitive, and $\Aut(H(d,k))=S_k\wr S_d$ is primitive on the vertex set of $H(d,k)$ if and only if $k\geqslant 3$ (see \cite[Section~9.2]{Brouwer-Book}).

Given a finite group $G$ and a subset $S\subseteq G\setminus\{1\}$, the {\em Cayley graph} $\Cay(G,S)$ of $G$ with respect to $S$ is the graph with vertex set $G$ and edge set $\{\{g,sg\}\mid g\in G,s\in S\}$. It is well known that a graph is isomorphic to a Cayley graph if and only if it has a group of  automorphisms acting regularly on its vertex set (see \cite[Lemma~16.3]{B}). 

\subsection{Primitive groups}
A permutation group $G$ on $\Ome$ is said to be {\em primitive} if $G$ is transitive on $\Ome$ and the only partitions of $\Ome$ preserved by $G$ are either the singleton subsets or the whole of $\Ome$. Let $G$ be a transitive permutation group on a set $\Ome$. An orbit $(u,v)^G$ of $G$ on $\Ome\times\Ome$ is called an {\it orbital} of $G$ on $\Ome$, and
$(v,u)^G$ is called the {\it paired orbital} of $(u,v)^G$.
An orbital  $(u,v)^G$ is called {\it self-paired} if it equals $(v,u)^G$ (equivalently if $(v,u) \in (u,v)^G$).
The orbital $(u,u)^G$ is called {\it trivial}, and the others are  non-trivial.
To an orbital $E$ we associate the digraph with vertex set $\Ome$ and arc set $E$,
called the {\em orbital digraph} for $E$, denoted by the pair $(\Omega, E)$. The orbital digraph for $E$ is a \emph{graph} if and only if $E$ is self-paired. Furthermore, we extend this notation to unions of orbitals, namely, $(\Omega, \cup E_i)$, which is called a generalised orbital digraph. Each digraph with vertex set $\Omega$ admitting $G$ is a generalised orbital digraph.
For a given point $u$, each orbital $E=(u,v)^G$ corresponds to an orbit of $G_u$ on $\Ome$, namely, $\{v \mid (u,v)\in E\}$, and we say that this orbit of $G_u$ is {\it self-paired} if the corresponding orbital  is self-paired.
The orbits of $G_u$ on $\Ome$ are called {\em suborbits} of $G$,
and their sizes are called the {\em subdegrees} of $G$. Hence, the number $r$ of   orbits of $G_u$ on $\Ome$ is equal to the number of  orbits of $G$ on $\Ome\times\Ome$, the rank of $G$.

The following result is due to D.G. Higman.

\begin{prop}{\rm\cite[Proposition~4.4]{Sims1967} or \cite[Theorem~3.2A]{DM}}\label{primitive-orbital}
A finite transitive permutation group $G$ on $\Ome$ is primitive if and only if every non-trivial orbital digraph of $G$ is connected.
\end{prop}

We will also need the following useful lemma.

\begin{lem}\label{lem:socles}
Let $G<H$ be finite primitive permutation groups with $\soc(G)\neq\soc(H)$. Let $X=\soc(H)G\leqslant H$. Then $\soc(X)=\soc(H)$ and $G<X$.
\end{lem}
\begin{proof}
First note that since $G$ is primitive, so is $X$. Let $N$ be a minimal normal subgroup of $H$. Suppose first that $N$ is abelian. Then $N=\soc(H)\norml X$. Let $M$ be a minimal normal subgroup of $X$ that is contained in $N$. Since $N$ is abelian, so is $M$ and so $M$ is the unique minimal normal subgroup of $X$ \cite[Theorem 4.3B]{DM}. Thus $M=C_X(M)$ and so $N=M= \soc(X)$. Next suppose that $N$ is nonabelian. Then $N=T_1\times T_2\times \cdots \times T_k$ where the $T_i$ are pairwise isomorphic simple groups.  Now each $T_i$ is a minimal normal subgroup of $N$ and $\soc(H)$. Moreover, $H$ and hence $G$, permutes the $T_i$. Thus for each $i$, the normal closure of $T_i$ in $G$ is a minimal normal subgroup of $X$ consisting of a direct product of the simple direct factors of $N$. Hence $N$ is the product of minimal normal subgroups of $X$ and so $N\leqslant \soc(X)$. Thus $\soc(H)\leqslant \soc(X)$. If $\soc(H)<\soc(X)$ then there exists a minimal normal subgroup $L$ of $X$ such that $L\cap \soc(H)=1$. Thus $[L,\soc(H)]=1$, contradicting the fact that $C_H(\soc(H))\leqslant \soc(H)$ \cite[Theorem 4.3B]{DM}. Hence $\soc(H)=\soc(X)$.

Finally, if $G=X$ then $\soc(G)=\soc(X)=\soc(H)$, a contradiction.
\end{proof}

\subsection{The subdegrees of the affine rank $3$ groups}

One of the main approaches in this paper is to use the subdegrees of the affine rank $3$ groups to
calculate the subdegrees of an affine primitive $2$-closed permutation group $G$ of rank $4$.
The following result of Liebeck on subdegrees of affine rank $3$ groups is instrumental in our investigation.

\begin{theorem}{\rm\cite[Appendix~2]{Liebeck-affine}}\label{subdegree}
Let $G$ be a finite primitive affine permutation group of rank $3$ and of degree $p^d$, with socle $V$, where $V\cong\mz_p^d$ for some prime $p$, and let $G_0$ be the stabiliser of the zero vector in $V$. Then the subdegrees of  $G$ are listed in Tables {\rm\ref{tab-sub-1}--\ref{tab-sub-3}}.
\end{theorem}

\begin{table}
\begin{center}
\caption{$G$ in class (A) of \cite[Theorem]{Liebeck-affine}.} \label{tab-sub-1}
\newcommand{\tabincell}[2]{\begin{tabular}{@{}#1@{}}#2\end{tabular}}
\begin{tabular}{|l|l|c|c|}
\hline Type of $G$ & $p^d$ &  subdegrees   \\
\hline
(A1): $G_0\leq\GammaL_1(p^d)$ & $p^d$ & \tabincell{l}{$\frac{p^d-1}{v},  \frac{(v-1)(p^d-1)}{v}$, where $v$ is a prime} \\
\hline
(A2): $G_0$ imprimitive & $p^{2m}$ & $2(p^m-1), (p^m-1)^2$ \\
\hline
\tabincell{l}{(A3)--(A5):\\
$G_0$ preserves $X\otimes Y$,\\
$\dim(X)=2$, \\
$\dim(Y)=m\geq2$} & $q^{2m}$ & $(q+1)(q^m-1),  q(q^m-1)(q^{m-1}-1)$ \\
\hline
(A6): ${\rm SU}_a(q)\unlhd G_0$& $q^{2a}$ & $\left\{
                                              \begin{array}{ll}
                                                (q^{a-1}+1)(q^a-1), q^{a-1}(q-1)(q^{a}-1), & a\ \hbox{even} \\
                                                (q^{a-1}-1)(q^a+1), q^{a-1}(q-1)(q^{a}+1), & a>1\ \hbox{odd}
                                              \end{array}
                                            \right.$\\
\hline
(A7): $\Omega_{2a}^\varepsilon(q)\unlhd G_0$& $q^{2a}$ & $\left\{
                                              \begin{array}{ll}
                                                (q^{a-1}+1)(q^a-1),  q^{a-1}(q-1)(q^{a}-1), & \varepsilon=+ \\
                                                (q^{a-1}-1)(q^a+1),  q^{a-1}(q-1)(q^{a}+1), & \varepsilon=-
                                              \end{array}
                                            \right.$\\
\hline
(A8): $\SL_5(q)\unlhd G_0$& $q^{10}$ & $(q^5-1)(q^2+1),  q^2(q^5-1)(q^{3}-1)$\\
\hline
(A9): $B_3(q)\unlhd G_0$& $q^{8}$ & $(q^4-1)(q^3+1),  q^3(q^4-1)(q-1)$\\
\hline
(A10): $D_5(q)\unlhd G_0$& $q^{16}$ & $(q^8-1)(q^3+1),  q^3(q^8-1)(q^5-1)$\\
\hline
(A11): ${\rm Sz}(q)\unlhd G_0$& $q^{4}$ & $(q^2+1)(q-1),   q(q^2+1)(q-1)$\\
\hline
\end{tabular}
\end{center}
\end{table}

\begin{table}
\begin{center}
\caption{$G$ in class (B) of \cite[Theorem]{Liebeck-affine}.}\label{tab-sub-2}

\begin{tabular}{|l|c|c|c||l|c|c|c|}
\hline $r$ & $p^d$ &  $R$  & subdegrees & $r$ & $p^d$ &  $R$  & subdegrees\\
\hline
3 & $2^6$ & $3^{1+2}$ & $27, \quad 36$ &2 & $31^{2}$ & $R^{1}_1$ or $R_2^1$ & $240, \quad 720$\\
\hline
2 & $3^{4}$ & $R^{1}_1$ or $R_2^1$ & $32, \quad 48$&2 & $47^{2}$ & $R^{1}_1$ or $R_2^1$ & $1104, \quad 1104$\\
\hline
2 & $7^{2}$ & $R^{1}_1$ or $R_2^1$ & $24, \quad 24$ &2 & $3^{4}$ & $R^{1}_2$ & $32, \quad 48$ \\
\hline
2 & $13^{2}$ & $R^{1}_1$ or $R_2^1$ & $72, \quad 96$ & 2 & $3^{4}$ & $R_2^2$ & $16a, \quad  16b \quad (a+b=5)$\\
\hline
2 & $19^{2}$ & $R^{1}_1$ or $R_2^1$ & $96, \quad 192$ & 2 & $5^{4}$ & $R_2^2$ & $240, \quad 384$\\
\hline
2 & $23^{2}$ & $R^{1}_1$ or $R_2^1$ & $144, \quad 216$ & 2 & $5^{4}$ & $R^{2}_3$ & $240, \quad 384$\\
\hline
2 & $3^{6}$ & $R^{1}_1$ or $R_2^1$ & $264, \quad 264$ &2 & $7^{4}$ & $R_2^2$ & $480, \quad 1920$ \\
\hline
2 & $29^{2}$ & $R^{1}_1$ or $R_2^1$ & $168, \quad 672$ &2 & $3^{8}$ & $R^{2}_3$ & $1440, \quad 5120$ \\
\hline

\end{tabular}
\end{center}
\end{table}

\begin{table}
\begin{center}
\caption{$G$ in class (C) of \cite[Theorem]{Liebeck-affine}.}\label{tab-sub-3}

\begin{tabular}{|l|c|c||l|c|c||l|c|c|}
\hline $L$ & $p^d$ &   subdegrees & $L$ & $p^d$ &   subdegrees&$L$ & $p^d$ &   subdegrees\\
\hline
$A_5$ & $3^4$ & $40,40$ & $A_6$ & $5^4$ & $144,480$  & $M_{11}$ & $3^5$ &  $22, 220$ or $110, 132$\\
\hline
$A_5$ & $31^2$ & $360,600$ & $A_7$ & $2^8$ & $45,210$ & $M_{24}$ & $2^{11}$ &  $276, 1771$ or $759, 1288$\\
\hline
$A_5$ & $41^2$ & $480,1200$ & $A_7$ & $7^4$ &  $720, 1680$&Suz & $3^{12}$ & $65520,465920$ \\
\hline
$A_5$ & $7^4$ & $960,1440$ & $A_9$ & $2^8$ &  $120, 135$& $G_2(4)$ & $3^{12}$ & $65520,465920$\\
\hline
$A_5$ & $71^2$ & $840,4200$ & $A_{10}$ & $2^8$ &  $45, 210$ &$J_2$ & $2^{12}$ & $1575,2520$\\
\hline
$A_5$ & $79^2$ & $1560,4680$ & $L_2(17)$ & $2^8$ &  $102, 153$&$J_2$ & $5^{6}$ & $7560,8064$\\
\hline
$A_5$ & $89^2$ & $2640,5280$ & $L_3(4)$ & $3^6$ &  $224, 504$&&&\\
\hline
$A_6$ & $2^6$ & $18,45$ & $U_4(2)$ & $7^4$ &  $240, 2160$&&&\\
\hline

\end{tabular}
\end{center}
\end{table}

\begin{note}
{\rm\begin{enumerate}
\item [{\rm (1)}] As shown in  \cite[Sections 11 and 12]{BDFP}, Liebeck's classes {\rm (A4)} and {\rm (A5)}  (classes {\rm (S1)} and {\rm (S0)} in {\rm \cite[Theorem 3.1]{BDFP}}) preserve a tensor decomposition $V=X\otimes Y$ with $\dim(X)=2$ so we group these classes together with {\rm (A3)} in Table \ref{tab-sub-1}.
\item [{\rm (2)}]  Table~\ref{tab-sub-2} lists the `Extraspecial class':
Here $G_0\leqslant N_{\GL_d(p)}(R)$, where $R$ is an $r$-group, irreducible on $V$. Either $r=2$ and $R\cong 3^{1+2}$ (extraspecial of order $27$), or $r=2$ and $|R/Z(R)|=2^{2m}$ with $m=1$ or $2$. If $r=2$, then either $|Z(R)|=2$ and $R$ is one of the two extraspecial groups $R_1^m$, $R_2^m$ of order $2^{1+2m}$, or $|Z(R)|=4$ and
we write $R=R^m_3$.
\item [{\rm (3)}] Table~\ref{tab-sub-2} lists the `Exceptional class'. Here $L$ is the simple socle of $G_0/(G_0\cap Z)$ where $Z=Z(\GL(a,q))$ with $q^a=p^d$. See \cite[Remark 3.4(1)]{BDFP}.
\item [{\rm (4)}] The subdegrees for the row corresponding to $L=A_9$ in Table~\ref{tab-sub-3} are recorded as $105$ and $150$ in  \cite[Table~14]{Liebeck-affine}; the correct ones should be $120$ and $135$ as given by \cite[Remark 3.4(3)]{BDFP}.
\end{enumerate}}
\end{note}

We will also need the following useful lemma.

\begin{lem}\label{lem:larger}
Let $G$ be one of the groups in Table~\ref{tab-sub-1}. Suppose further that $p^d\geqslant 4096$ and that $G$ belongs to types {\rm(A3)--(A11)}. Let $m_1$ and $m_2$ be the non-trivial subdegrees of $G$. Then exactly one of $m_1$ and $m_2$ is divisible by $p$. Moreover, if $m_2$ is the subdegree divisible by $p$ then either:
\begin{enumerate}
\item [{\rm (1)}] $m_1<m_2$, or
\item [{\rm (2)}] $q=2$, $G$ has type {\rm(A6)} or {\rm(A7)},  and $m_1=(2^{a-1}+1)(2^{a}-1)>m_2=2^{a-1}(2^a-1)$.
\end{enumerate}
\end{lem}
\begin{proof} We work through each of the types. Suppose first that  $G$ has type (A3)--(A5).  Then  $m_1=(q+1)(q^m-1)$ and $m_2=q(q^{m-1}-1)(q^m-1)$ with $m>1$.   Thus $q(q^{m-1}-1)-(q+1)=q^m-2q-1=q(q^{m-1}-2)-1>0$ since $q^{2m}\geqslant 4096$. Hence $m_1<m_2$ in this case.

Next suppose that $G$ has type (A8).  Then  $m_1=(q^2+1)(q^5-1)$ and $m_2=q^2(q^{3}-1)(q^5-1)$.   Thus $q^2(q^{3}-1)-(q^2+1)=q^5-2q^2-1=q^2(q^{3}-2)-1>0$, and so we again have that $m_1<m_2$.

If $G$ has type (A9) then $p^d=q^{10}$. Moreover,  $m_1=(q^3+1)(q^4-1)$ and $m_2=q^3(q-1)(q^4-1)$.  Then $q^3(q-1)-(q^3+1)=q^4-2q^3-1=q^3(q-2)-1>0$ since $q^{10}\geqslant 4096$. Hence $m_1<m_2$.

Next suppose that $G$ has type (A10). Then  $m_1=(q^3+1)(q^8-1)$ and $m_2=q^3(q^5-1)(q^8-1)$. Thus $q^3(q^5-1)-(q^3+1)=q^8-2q^3-1=q^3(q^5-2)-1>0$ and so we again obtain that $m_1<m_2$.

If $G$ belongs to type (A11), then we have $m_1=(q^2+1)(q-1)$ and $m_2=q(q-1)(q^2+1)$.   Thus $q(q-1)>(q-1)$ and so $m_1<m_2$.

Finally, suppose that $G$ has  type (A6) or (A7). Then $p^d=q^{2a}$ with $a>1$. If $m_1=(q^{a-1}-1)(q^a+1)$ and $m_2=q^{a-1}(q-1)(q^a+1)$, then $m_2-m_1=(q^{a-1}(q-2)+1)(q^a-1)>0$. Thus we again deduce that $m_2>m_1$. On the other hand, if $m_1=(q^{a-1}+1)(q^a-1)$ and $m_2=q^{a-1}(q-1)(q^a-1)$, then $m_2-m_1=(q^{a-1}(q-2)-1)(q^a-1)$. Thus for $q>2$ we obtain the usual conclusion that $m_2>m_1$. However, for $q=2$ we have  $m_2<m_1$ and we obtain the exceptional case in the statement of the lemma.
\end{proof}

\subsection{ Rank 4 groups}

We begin with the following proposition.
\begin{prop}
\label{prn:rank4}
Let $G$ be a primitive permutation group on $\Omega$ of rank $4$. Then one of the following holds:
\begin{enumerate}[{\rm (1)}]
  \item $G$ is affine, that is, $\soc(G)$ is abelian.

  \item  $G$ is almost simple, that is, $\soc(G)$ is nonabelian simple.

  \item $\PSL_2(8)^2 \lhd G\leqslant G_0\wr S_2$, and $\Ome=\Delta^2$,  where $|\Delta|=28$, $\PSL_2(8)=\soc(G_0)$ and  $G$ acts transitively on the simple direct factors of $\PSL_2(8)^2$.

  \item  $T^3\lhd G\leqslant G_0\wr S_3$, and $\Ome=\Delta^3$,
  where $G_0$ is a $2$-transitive group on $\Delta$,  and $G$ acts transitively on the simple direct factors of $T^3$.
  

  \item  $\soc(G)=T\times T$ with $T=\A_5$,
 and  the point-stabiliser $\soc(G)_\a$ is the diagonal subgroup of $T\times T$.
 \end{enumerate}
\end{prop}
\begin{proof}
Assume that $G$ is neither affine nor almost simple. Then by \cite{Libeck-Praeger-Saxl-1988}, the socle of $G$ is $T^k$, where $T$ is a non-abelian simple group and $k\geqslant 2$, and $G$ has either simple diagonal action, product action or twisted wreath action on $\Omega$ (for the details of these actions see \cite[Section~1]{Libeck-Praeger-Saxl-1988}). Since $G$ has rank $4$, by \cite[Proposition~5.1]{Cameron}, we have $k\leqslant 3$. If $G$ has twisted wreath action, then we have $k\geq 6$ (see \cite[Theorem~4.7B(iv)]{DM}), a contradiction. If $G$ has product action, then by \cite[Theorem~1(P)]{Vauhkonen-phd-thesis},  part (3) or (4) happens.  If $G$ has simple diagonal action, then by \cite[Theorem~1(D)]{Vauhkonen-phd-thesis}, part (5) happens. The result then follows.
\end{proof}

%
%
%
%
%
%

We also need a result about the containment of rank 4 groups inside rank 3 groups.

\begin{prop}\label{prn:4in3}
Let $G$ be an almost simple primitive group of rank $4$ on a set $\Omega$ of size $n\geqslant 4096$. Suppose that $G$ is contained in a permutation group $H$ of rank $3$ with $\soc(G)\neq \soc(H)$ and $H=\soc(H)G$. Then $G$ and $H$ have a common nontrivial subdegree $d$ and one of the following holds:
\begin{enumerate}[{\rm (a)}]
    \item $\soc(G)=G_2(q)$, $\soc(H)=\Omega_7(q)$, $n=\frac{q^6-1}{q-1}$ and $d=q^5$; or
    \item $\soc(G)=\Omega_7$, $\soc(H)=\POmega_8^+(q)$, $n=\frac{(q^4-1)(q^3+1)}{q-1}$ and $d=q^6$.
\end{enumerate}
\end{prop}
\begin{proof}
Let $\alpha\in\Omega$. Note that since $G_\alpha$ leaves invariant the three orbits of $H_\alpha$, it follows that only one of the two  orbits of $H_\alpha$ on $\Omega\backslash\{\alpha\}$ splits into two $G_\alpha$-orbits. Thus $G$ and $H$ have a common nontrivial orbital with corresponding suborbit size $d$. Thus $\soc(G)$,  $\soc(H)$ and $d$ are given by \cite[Theorem~1]{Libeck-Praeger-Saxl2}. It remains to check which pairs $(G,H)$  are rank 3 and rank 4 groups. This and the degree assumption that $n\geqslant 4096$ rules out all the possibilities in \cite[Tables 1 and 7.2]{Libeck-Praeger-Saxl2} except for possibly one of the triples 
 $(\soc(G),\soc(H),n)=(\Sp_{4}(4),\Sp_{4}(4)\times \Sp_{4}(4),14400)$,
    $(A_{23},A_{24}, |A_{24}:M_{24}|)$, $(M_{23},M_{24}, 40320)$, $(\Omega_7(3),\POmega_8^+(3),28431)$ or $(\Sp_{22}(2),\Omega^+_{24}(2),|\Omega_{24}^+(2):\mathrm{Co}_1|).$
For the first we see from \cite{BCP} that $G$ has rank greater than 4, while for the latter we use \cite{Bannai}, \cite{LSrank3} and \cite{K-L-rank-3Litype} to deduce that rank$(H)>3$. This eliminates these cases.

If $G$ and $H$ are given by \cite[Theorem~1(c)]{Libeck-Praeger-Saxl2}, then $G<H\leqslant \Sym(\Omega_0)\wr S_m$ acting in product action on $\Omega=\Omega_0^m$, for some $m\geqslant 2$ and there exist $G_0\leqslant H_0\leqslant\Sym(\Omega_0)$, such that $H_0=\soc(H_0)G_0$ and $G_0$ and $H_0$ share a common non-trivial orbital in $\Omega_0\times \Omega_0$.  Moreover, as discussed in \cite[Section 3]{Libeck-Praeger-Saxl2}, we have that $G\leqslant G_0\wr S_m$ and $H\leqslant H_0\wr S_m$.
Since rank$(H)=3$, we must have that $m=2$ and $H_0$ is 2-transitive on $\Omega_0$. Then as $G$ has rank 4 it follows from Proposition \ref{prn:rank4}, that $\soc(G_0)=\PSL_2(8)$ and $|\Omega_0|=28$. This contradicts the fact that $n\geqslant 4096.$ 

It remains to consider the infinite families in \cite[Table 7.1]{Libeck-Praeger-Saxl2}. For the groups in lines~1--4 of \cite[Table 7.1]{Libeck-Praeger-Saxl2}, $\soc(H)$ is an alternating group, and noting that $|\Omega|\geqslant 4096$, by \cite{Bannai}, we see that rank$(H)>3$, a contradiction. 

For the groups in line~5 of \cite[Table 7.1]{Libeck-Praeger-Saxl2}, we have $\soc(G)=L_2(q)$, $\soc(H)=A_{q+1}$ and $\Omega$ is the set of pairs of $\{1,\ldots, q+1\}$. In this case, both $G$ and $H$ are $3$-transitive on $\{1,\ldots,q+1\}$, and $H$ does have rank 3. However, for $n\geqslant 4096$, $G$ does not have rank $4$ (see \cite[Theorem 1.1]{Cuypers} or \cite[Theorem 4]{Vauhkonen-phd-thesis}), a contradiction.

For the groups in line~6 of \cite[Table 7.1]{Libeck-Praeger-Saxl2}, we have $\soc(G)=G_2(q)$ and $\soc(H)=\Omega_7(q)$. If $\Omega=(H: P_1)$, then rank$(H)=3$. Moreover, the action of $G$ on $\Omega$ is the action of $G$ on the set of points of the associated generalised hexagon which has rank 4. Thus we obtain case (a).  If $\Omega=(H: N_1^-)$ or $(H: N_1^+)$, then since $|\Omega|\geqslant 4096$, by  \cite[Lemma 6.8]{Libeck-Praeger-Saxl2}, we have rank$(H)>3$, a contradiction.

For the groups in line~7 of \cite[Table 7.1]{Libeck-Praeger-Saxl2}, we have $\soc(G)=\Omega_7(q)$ and $\soc(H)=\POmega_8^+(q)$. Then by \cite[Lemma 6.7]{Libeck-Praeger-Saxl2}, rank$(H)>3$ unless $q=3$, but in this case $\rank(G)=\rank(H)$, a contradiction.
 
For the groups in line~8 of \cite[Table 7.1]{Libeck-Praeger-Saxl2}, we have $\soc(G)=\Omega_{2m-1}(q)$ and $\soc(H)=\POmega_{2m}^+(q)$. If $\Omega=(H: P_m)$ or $(H:P_{m-1})$ then the action of $H$ is on one of the orbits of $H$ on maximal totally singular subspaces. Since rank$(H)=3$, it follows from \cite{K-L-rank-3Litype} that  $m=4$. Thus the three actions on $(H: P_m)$, $(H:P_{m-1})$ and $(H:P_1)$ are permuted by triality and so are permutationally isomorphic. Hence we may assume that $\Omega=(H:P_4)$ and $G$ is the stabiliser in $H$ of a nondegenerate hyperplane $W$. Then by \cite[pg.~63]{LPSmaxfacts}, $H_\alpha\cap G$ is the stabiliser in $G$ of a maximal totally singular subspace of $W$.   Thus the action of $G$ on $\Omega$ is equivalent to the action of $G$ on maximal totally singular subspaces of dimension 3 in $W$. This action has rank 4 \cite[Theorem~9.4.3]{Brouwer-Book} and so we get case (b), where the subdegree is given by \cite[Lemma~9.4.1]{Brouwer-Book}.

For the groups in line~9 of \cite[Table 7.1]{Libeck-Praeger-Saxl2}, we have $G=L_{m}(2).2$ (where $G$ contains a graph automorphism) and 
$H=\Omega_{2m}^+(2).(2,m-1)$,
with $m\geqslant 4$. In this case, we have $\Omega=(H: N_1)$ and as seen in the proof of \cite[Lemma 6.5]{Libeck-Praeger-Saxl2},  $H$ has rank 3. Moreover, the action of $G$ on $\Omega$ is the action of $G$ on the set of decompositions $V=\langle e_1\rangle\oplus\langle e_2\ldots,e_m\rangle$ of an $m$-dimensional vector space over $\GF(2)$.  Note that if $\langle u\rangle \oplus U$ and $\langle v\rangle \oplus W$ lie in the same suborbit of the stabiliser of the decomposition $V=\langle e_1\rangle\oplus\langle e_2\ldots,e_m\rangle$, then $\dim(U\cap \langle e_2\ldots,e_m\rangle) =\dim(W\cap \langle e_2\ldots,e_m\rangle)$. Since $|\Omega|\geqslant 4096$, we have that $m>5$ and so   this action has rank at least 5, a contradiction.

For the groups in lines~10--13 of \cite[Table 7.1]{Libeck-Praeger-Saxl2},  by \cite{K-L-rank-3Litype}, we have rank$(H)>3$, a contradiction. %
   
For the groups in lines~14 of \cite[Table 7.1]{Libeck-Praeger-Saxl2},  we have that $\soc(G)=\Sp_m(16)$ and $\soc(H)=\Sp_{2m}(4)$ and $\Omega=(\soc(H): O_{2m}^{\pm}(4))$. By  \cite[Lemma~6.2]{Libeck-Praeger-Saxl2}, we have that rank$(H)=3$. Moreover, by \cite[3.2.1(d)]{LPSmaxfacts} we see that $\soc(G)_\alpha=O_{2m}^{\pm}(16)$. The rank of this action is given in the proof of \cite[Lemma~6.2]{Libeck-Praeger-Saxl2}, where we see that it has rank 8, a contradiction.
\end{proof}

\subsection{$\frac{3}{2}$-transitive permutation group}

Let $G$ be a transitive permutation group on a set $\Ome$.
If $G$ is non-regular on $\Omega$ and all the nontrivial orbits of a point-stabiliser have equal size, then we say that $G$ is a {\em $\frac{3}{2}$-transitive permutation group} on $\Omega$. The following proposition provides a classification of $\frac{3}{2}$-transitive permutation groups.

\begin{prop}{\rm\cite[Corollary~3]{Libeck-Praeger-Saxl3}}\label{2/3-trans}
Let $G$ be a $\frac{3}{2}$-transitive permutation group of degree $n$. Then one of the following holds:
\begin{enumerate}
  \item [{\rm (1)}]\ $G$ is $2$-transitive;
  \item [{\rm (2)}]\ $G$ is a Frobenius group;
  \item [{\rm (3)}]\ $G$ is affine and $n=p^d$ with $p$ prime: $G=\mz_p^d\rtimes H$, where $H\leqslant \GL_d(p)$ is one of the following groups:
\begin{enumerate}
  \item [{\rm (i)}]\ $H\leqslant \GammaL_1(p^d)$
  \item [{\rm (ii)}]\ $H=S_0(p^{d/2})$ is the subgroup of $\GL_2(q)$ of order $4(q-1)$ consisting of all monomial matrices of determinant $\pm1$,
 with $p$ an odd prime;
  \item [{\rm (iii)}]\ $H$ is soluble and $n=p^d=3^2,5^2,7^2,11^2,17^2$ or $3^4$;
  \item [{\rm (iv)}]\ $\SL_2(5)\leqslant H\leqslant \GammaL_2(p^{d/2})$, where $p^{d/2}=9, 11, 19, 29$ or $169$.
\end{enumerate}
  \item [{\rm (4)}]\ $G$ is almost simple: either
\begin{enumerate}
  \item [{\rm (a)}]\ $n=21$, $G=A_7$ or $S_7$ acting on the set of pairs in $\{1,\ldots, 7\}$, or
  \item [{\rm (b)}]\  $n=\frac{1}{2}q(q-1)$ where $q=2^f\geqslant 8$, and either $G=\PSL_2(q)$, or $G={\rm P\G L}_2(q)$ with $f$ a  prime.

\end{enumerate}
\end{enumerate}

\end{prop}

We shall end this section with the following lemma.

\begin{lem}\label{lem-3/2}
Let $G=\mz_p^d\rtimes H$ be a $\frac{3}{2}$-transitive permutation group on $V=\mz_p^d$, where $1\neq H\leq\GL_d(p)$ and $p$ is a prime.
If $H$ has three orbits on $V\setminus\{0\}$, then $G$ is primitive.
\end{lem}

\begin{proof} Note that each orbit of $H$ on $V\backslash\{0\}$ has length $\frac{(p^d-1)}{3}$. Suppose that $G$ is imprimitive. By Proposition~\ref{primitive-orbital},
there exists one non-trivial orbital graph, say $\G$ of $G$, which is disconnected. Then $\G$ would have a connected component of order $\frac{(p^d-1)}{3}+1$ or $\frac{2(p^d-1)}{3}+1$, and so either $\frac{(p^d-1)}{3}+1$ or $\frac{2(p^d-1)}{3}+1$ divides
$p^d$. For the former, we have $\frac{(p^d-1)}{3}+1=\frac{p^d+2}{3}$ divides $p^d$, and so $3p^d=k(p^d+2)$ for some positive integer $k$. It follows that $2k=(3-k)p^d$, and hence $k=2$ and $p^d=4$. This forces $H=1$, a contradiction. For the latter, we have $\frac{2(p^d-1)}{3}+1=\frac{2p^d+1}{3}$ divides $p^d$, and so $3p^d=k(2p^d+1)$  for some integer $k$. It follows that $k=(3-2k)p^d$, and hence $k=1$, a contradiction. Thus $G$ is primitive.\end{proof}

\section{Examples}\label{examples}

\subsection{Tensor products}\label{sec:tensor}
We begin with the following useful lemma.

\begin{lem}\label{lem:hamming}
Let $G_0\leqslant\GL(V)$ preserve the decomposition $V=V_1\oplus V_2$ and let $G=V\rtimes G_0\leqslant \AGL(V)$. Suppose that  $\dim(V_1)=\dim(V_2)$ and that  $B=(V_1\cup V_2)\backslash\{0\}$ is an orbit of $G_0$. Then the orbital digraph $\Gamma$ for $G$ arising from the suborbit $B$ is isomorphic to the Hamming graph $H(2,|V_1|)$.
\end{lem}
\begin{proof}
Since $V=V_1\oplus V_2$ the elements of $V$ can be identified with ordered pairs whose first coordinate lies in $V_1$ and whose second coordinate lies in $V_2$. With this identification $V_1$ is identified with those elements of $V$ whose second coordinate is $0$ and $V_2$ is identified with those elements of $V$ whose first coordinate is $0$. The neighbours of $(0,0)$ in $\Gamma$ are then those vertices which differ in exactly one coordinate.  Since $V$ acts transitively by addition on the set of vertices of $\Gamma$ and induces automorphisms, it follows that two vertices are adjacent if and only if they differ in exactly one coordinate. Hence $\Gamma$ is the Hamming graph $H(2,|V_1|)$.
\end{proof}

We can now give two infinite families of 2-closed groups of rank 4 that are not the automorphism group of any graph or digraph.

\begin{lem}\label{family2}
Let $m\geqslant 2$ and let $X$ and $Y$ be $2$- and $m$-dimensional $\GF(3)$-spaces, respectively. Let  $D_8$ be the subgroup of $\GL_{2}(3)$ stabilising a decomposition of $X$ into the direct sum of two $1$-dimensional subspaces  and let $D_8 \circ \GL_m(3)$ act on $V=X\otimes Y$.   Let $\GG(m)=V\rtimes (D_8\circ \GL_m(3))$ be the permutation group with $V$ as a regular normal subgroup. Then $\GG(m)$  is a $2$-closed primitive affine permutation group of degree $3^{2m}$ and of rank $4$. Moreover, $\GG(m)$ is not the automorphism group of any graph or digraph.
\end{lem}

\begin{proof} 
 For convenience, we set $\GG=\GG(m)$ and we let $D_8$ be the subgroup of $\GL_{2}(3)$ preserving the decomposition $X=\langle x_1\rangle\oplus \langle x_2\rangle$. Then $\GG_0=D_8\circ\GL_m(3)$. Let $M=\GL_{2}(3)\circ\GL_m(3)$ be the stabiliser in $\GL_{2m}(3)$ of the tensor decomposition $V=X\otimes Y$. 

Let $V_1=\langle x_1\rangle\otimes Y$ and $V_2=\langle x_2\rangle\otimes Y$. Then $V=V_1\oplus V_2$ and $\dim(V_1)=\dim(V_2)=m$. Moreover, $\GG_0$ is the stabiliser in $M$ of this decomposition $V=V_1\oplus V_2$.  Let \[B_1=(V_1\cup V_2)\setminus\{0\}\quad\quad  B_2=(\langle x_1+x_2\rangle\otimes Y\cup \langle x_1-x_2\rangle\otimes Y)\setminus\{0\}\]
and let $B_3=V^*\backslash(B_1\cup B_2)$.  By \cite[Lemma~1.1]{Liebeck-affine}, $B_1\cup B_2$ and $B_3$ are the two orbits of $M$ on $V^*$. It is also easy to see that $B_1$ and $B_2$ are $\mathcal{G}_0$-orbits. Since $V=V_1\oplus V_2$, Lemma \ref{lem:hamming} implies that the orbital digraph for $\mathcal{G}$ arising from $B_1$ is the Hamming graph $H(2,3^m)$.  Similarly, $\GG_0$ stabilises the decomposition $V=W_1\oplus W_2$ where $W_1=\langle x_1+x_2\rangle\otimes Y$ and $W_2= \langle x_1-x_2\rangle\otimes Y$. Thus Lemma \ref{lem:hamming} implies that the orbital graph of $\GG$ arising from  $B_2$ is also isomorphic to $H(2,3^m)$. Note that the automorphism group of $H(2,3^m)$ is $S_{3^m}\wr S_2$, which is larger than $\GG_0$.

Now the vectors of $V$ that are in neither $V_1,V_2,W_1$ nor $W_2$ are of the form $\lambda_1x_1\otimes u+\lambda_2x_2\otimes v$ where $\lambda_1,\lambda_2\neq 0$ and $\langle u,v\rangle$ is a 2-dimensional subspace of $Y$. Since $\GL_m(3)$ acts transitively on the set of 2-dimensional subspaces of $Y$ and $\GL_m(3)_{\langle u,v\rangle}$ induces $\GL_{2}(3)$ on $\langle u,v\rangle$, it follows that $\GG_0$ acts transitively on $B_3$.
Hence $B_3$ is a common orbit of $\GG_0$ and $M$, and hence $M$ is contained in the automorphism group of the orbital graph of $\GG$ arising from $B_3$. Thus $\mathcal{G}$ is not the automorphism group of either of its three nontrivial orbital digraphs. 

If $m=2$ we can check in \textsc{Magma} that $\GG$ is 2-closed, so assume that $m\geqslant 3$. Let $\mathcal{R}=V\rtimes M$. Then $\GG^{(2)}\leqslant \mathcal{R}^{(2)}$. Since $\GG_0$ is the stabiliser in $\mathcal{R}_0$ of $B_1$, to prove  that $\GG^{(2)}=\GG$, it suffices to show that $\mathcal{R}=\mathcal{R}^{(2)}$. As noted above, $\mathcal{R}$ has rank three and the orbits of $\mathcal{R}_0$ are $\{0\}$, $B_1\cup B_2$ and $B_3$. Thus $\mathcal{R}^{(2)}$ also has rank 3. Suppose, for a contradiction,  that we can identify $V$ with a cartesian product $V=\Delta^k$ for some set $\Delta$ and integer $k\geqslant 2$ such that $\mathcal{R}^{(2)}\leqslant \Sym(\Delta)\wr S_k$ in product action. Since $\mathcal{R}^{(2)}$ has rank three, we must have that $k=2$, (see for example \cite{LSrank3}) and $\mathcal{R}^{(2)}=\Sym(\Delta)\wr S_2$. Moreover, the orbits of $(\mathcal{R}^{(2)})_0$ have length 1, $2(|\Delta|-1)$ and $(|\Delta|-1)^2$. Since $|V|=3^{2m}$ we must have that $|\Delta|=3^m$. However, the orbits of $\mathcal{R}_0$ and hence of $(\mathcal{R}^{(2)})_0$ have length 1, $4(3^m-1)=4(|\Delta|-1)$ and $3^{2m}-4(3^m-1)-1=(|\Delta|-3)(|\Delta|-1)$, a contradiction as $|\Delta|\neq 5$. Thus $\mathcal{R}^{(2)}\nleqslant \Sym(\Delta)\wr S_k$ and so by \cite[Lemma~4.1]{PraegerSaxl} it follows that $\mathcal{R}^{(2)}\leqslant \AGL_{2m}(3)$. Now \cite[Lemma 4.4.5(1)]{XPLG}  implies that $(\mathcal{R}_0)^{(1)}=\mathcal{R}_0$, as $\mathcal{R}_0$ is the stabiliser in $\GL_{2m}(3)$ of a tensor product decomposition. Thus $\mathcal{R}^{(2)}=\mathcal{R}$ and the result follows.
\end{proof}

\begin{lem}\label{family3}
Let $m\geqslant 2$ and let $X$ and $Y$ be $2$- and $m$-dimensional $\GF(4)$-spaces, respectively. Let  $C_3\wr S_2$ be the subgroup of $\GL_{2}(4)$ stabilising a decomposition of $X$ into the direct sum of two $1$-dimensional subspaces  and let $((C_3\wr S_2) \circ \GL_{m}(4)).2\leqslant\GL_{2m}(4)$ act on $V=X\otimes Y$.  Let $\mathcal{H}(m)=V\rtimes (((C_3\wr S_2)\circ \GL_{m}(4)).2)$ be the permutation group with $V$ as a regular normal subgroup. Then $\mathcal{H}(m)$  is a $2$-closed primitive affine permutation group of degree $4^{2m}$ and of rank $4$. Moreover, $\mathcal{H}(m)$ is not the automorphism group of any graph or digraph.
\end{lem}
\begin{proof} 
 For convenience, we set $\mathcal{H}=\mathcal{H}(m)$ and we let $C_3\wr S_2$ be the subgroup of $\GL_{2}(4)$ preserving the decomposition $X=\langle x_1\rangle\oplus \langle x_2\rangle$. Then $\mathcal{H}_0=((C_3\wr S_2)\circ\GL_{m}(4)).2 \cong \GL_2(2)\times\GammaL_m(4)$. Let $M=(\GL_{2}(4)\circ\GL_{m}(4)).2$, which for $m\geqslant 3$ is the stabiliser in $\GammaL_{2m}(4)$ of the tensor decomposition $V=X\otimes Y$, and for $m=2$, $M$ is the index two subgroup of the stabiliser of the tensor decomposition that does not interchange $X$ and $Y$.

Let $V_1=\langle x_1\rangle\otimes Y$ and $V_2=\langle x_2\rangle\otimes Y$. Then $V=V_1\oplus V_2$ and $\dim(V_1)=\dim(V_2)=m$. Moreover, $\mathcal{H}_0$ is the stabiliser in $M$ of this decomposition $V=V_1\oplus V_2$.  Let \[B_1=(V_1\cup V_2)\setminus\{0\}\quad\quad  B_2=(\langle x_1+x_2\rangle\otimes Y\cup \langle x_1+\lambda x_2\rangle\otimes Y\cup \langle x_1+\lambda^2 x_2\rangle\otimes Y )\setminus\{0\},\]
where $\lambda$ is a primitive element of $\GF(4)$,
and let $B_3=V^*\backslash(B_1\cup B_2)$.  By \cite[Lemma~1.1]{Liebeck-affine}, $B_1\cup B_2$ and $B_3$ are the two orbits of $M$ on $V^*$. It is also easy to see that $B_1$ and $B_2$ are $\mathcal{H}_0$-orbits. Since $V=V_1\oplus V_2$, Lemma \ref{lem:hamming} implies that the orbital digraph for $\GG$ arising from $B_1$ is the Hamming graph $H(2,4^m)$. Note that the automorphism group of $H(2,4^m)$ is $S_{4^m}\wr S_2$, which is larger than $\mathcal{H}_0$.

Now the vectors of $V\backslash (B_1\cup B_2)$ are of the form $\lambda_1x_1\otimes u+\lambda_2x_2\otimes v$ where $\lambda_1,\lambda_2\neq 0$ and $\langle u,v\rangle$ is a 2-dimensional subspace of $Y$. Since $\GL_{m}(4)$ acts transitively on the set of 2-dimensional subspaces of $Y$ and $\GL_{m}(4)_{\langle u,v\rangle}$ induces $\GL_{2}(4)$ on $\langle u,v\rangle$, it follows that $\mathcal{H}_0$ acts transitively on $B_3$.
Hence $B_3$ is a common orbit of $\mathcal{H}_0$ and $M$. In particular, $\mathcal{H}_0$ has rank 4 and  $M$ is contained in the automorphism group of the orbital graph of $\mathcal{H}$ arising from $B_3$.

Note that $Y$ is a $2m$-dimensional vector space over $\GF(2)$. Let $u_1=x_1+x_2$, $u_2=\lambda x_1+\lambda^2x_2$ and $u_3=\lambda^2x_1+\lambda x_2$. Then $u_2\in\langle x_1+\lambda x_2\rangle$ and $u_3\in\langle x_1+\lambda^2x_2\rangle$. Moreover, $u_1+u_2=u_3$. Thus $\langle u_1,u_2,u_3\rangle_{\GF(2)}$ has dimension two. Note that for each $w\in Y$ and $i\in\{1,2,3\}$, we have that $\lambda u_i\otimes w=u_i\otimes (\lambda w)$. Hence for each $i\in\{1,2,3\}$ we have that  $\langle u_i\rangle_{\GF(2)}\otimes_{\GF(2)} W$ is a $2m$-dimensional vector space over $\GF(2)$ which  contains the same set of vectors as the $\GF(4)$-space $\langle u_i\rangle_{\GF(4)}\otimes_{\GF(4)} W$. Thus $B_2$ is the set of simple vectors of the tensor decomposition $V=\langle u_1,u_2\rangle_{\GF(2)}\otimes_{\GF(2)} W$. This decomposition is preserved by $\GL_{2}(2)\otimes \GL_{2m}(2)$ and so $\GL_{2}(2)\otimes\GL_{2m}(2)\leqslant \Aut(\Gamma_2)_0$.
Thus $\mathcal{H}$ is not the automorphism group of either of its three nontrivial orbital digraphs. 

When $m=2$ we can check in \textsc{Magma} that $\mathcal{H}$ is 2-closed, so suppose that $m\geqslant 3$.  Let $\mathcal{R}=V\rtimes M\geqslant \mathcal{H}$ and so $\mathcal{H}^{(2)}\leqslant \mathcal{R}^{(2)}$. As noted above, $\mathcal{R}$ has rank three and the orbits of $\mathcal{R}_0$ are $\{0\}$, $B_1\cup B_2$ and $B_3$. Thus $\mathcal{R}^{(2)}$ also has rank 3. The same argument as in the last paragraph of the proof of Lemma \ref{family2} comparing the orbit lengths of $\mathcal{R}_0$ to the orbits lengths of a rank three group in product action allows us to again deduce that $\mathcal{R}^{(2)}\nleqslant \Sym(\Delta)\wr S_k$ in product action with $V$ identified with $\Delta^k$ for some set $\Delta$ and $k\geqslant 2$. Thus \cite[Lemma~4.1]{PraegerSaxl} implies that $\mathcal{R}^{(2)}$ has the same socle as $\mathcal{R}$, that is, an elementary abelian group of order $2^{4m}$. Hence $\mathcal{R}^{(2)}\leqslant \AGL_{4m}(2)$.  Note that $\mathcal{H}_0$, and hence $(\mathcal{H}^{(2)})_0$, fix setwise $B_2$, the set of simple vectors of the tensor product decomposition $V=U\otimes W$ where $U$ is a 2-dimensional vector space over $\GF(2)$ and $W$ is a $2m$-dimensional vector space over $\GF(2)$. Thus by \cite[Lemma 4.4.5(1)]{XPLG} we have that $(\mathcal{H}^{(2)})_0\leqslant \GL_{2}(2)\times \GL_{2m}(2)$. By \cite{Dye}, $\GammaL_{m}(4)$ is a maximal subgroup of $\GL_{2m}(2)$. Hence $\mathcal{H}_0=\GL_{2}(2)\times\GammaL_{m}(4)$ is maximal in $\GL_{2}(2)\times \GL_{2m}(2)$. Since $\mathcal{H}_0$ has three orbits on $V\backslash\{0\}$ while $\GL_{2}(2)\times \GL_{2m}(2)$ has only two, it follows that $(\mathcal{H}^{(2)})_0=\mathcal{H}_0$ and so $\mathcal{H}$ is 2-closed.
\end{proof}

\subsection{$1$-dimensional semilinear}\label{sec:eg1diml}

Let $\omega$ be a primitive element of the field GF$(p^d)$ of order $p^d$, and let $\a: x\mapsto x^p$ be a generator of $\Aut(\GF(p^d))$. Then $\GammaL_1(p^d)=\langle\omega\rangle\rtimes\langle\a\rangle$. Note that $\GF(p^d)$ is a $d$-dimensional vector space over $\GF(p)$ and we can view $\GammaL_1(p^d)$ as a subgroup of $\GL_d(p)$. 

By \cite[Lemma~2.1]{Foulser-Kallaher-solvable-rank-3}, any subgroup $H$ of  $\GammaL_1(p^d)$ can be written uniquely in the form
 $H=\langle\o^m,\o^e\a^s\rangle$,
where the integers $m$, $e$ and $s$ satisfy the following conditions:
\begin{equation}\label{standard-condition}
m\ |\ p^d-1,\quad \quad s \mid d, \quad \quad e\left (\frac{p^d-1}{p^s-1}\right )\equiv0\ \pmod m.
\end{equation}
A subgroup $\langle\o^m,\o^e\a^s\rangle$ of $\langle\o,\a\rangle$ is said to be  in {\em standard form} if
$m$, $e$ and $s$ satisfy the conditions in Eq.~\eqref{standard-condition}.

The following proposition is directly obtained from {\rm\cite[Theorems~3.7, 3.8]{Foulser-Kallaher-solvable-rank-3}}. We let $V^*=\GF(p^d)\backslash\{0\}$.

\begin{prop}\label{lem-r3}
Let $H=\langle\o^m,\o^e\a^s\rangle$ (in standard form) have exactly two orbits on $V^*$, of lengths $\frac{m_1(p^d-1)}{m}$ and $\frac{m_2(p^d-1)}{m}$, where $m_1<m_2$. Then $m=vm_1$, $m_2=(v-1)m_1$, and $m_1, s$ are odd, and the following hold.
\begin{enumerate}
  \item [{\rm (1)}]\ each prime divisor of $m_1$ divides $p^s-1$;
  \item [{\rm (2)}]\ $v$ is an odd prime, and $v\mid (p^{sm_1(v-1)}-1)$ but $v\nmid (p^{sm_1t}-1)$ for $1\leqslant t<v-1$;
  \item [{\rm (3)}]\ {\rm gcd}$(e,m_1)=1$;
  \item [{\rm (4)}]\ $m_1s(v-1)$ divides $d$.
  \end{enumerate}
Conversely, assume that $e$ is an integer, and that $m_1$, $s$ and $d$ are positive integers satisfying the above conditions $(1)$--$(4)$. Let $m=vm_1$. Then the group $H=\langle \o^m,\o^e\a^s\rangle$ is in standard form and has exactly two orbits on $V^*$, of lengths $\frac{m_1(p^d-1)}{m}$ and $\frac{m_2(p^d-1)}{m}$.
\end{prop}

\begin{defi}\label{con}
Let $H=\langle \o^m, \o^e\a^s\rangle\leqslant\GammaL_1(p^d)$ be in standard form with $m=3m_1$ for some odd integer $m_1$ with $\gcd(m_1,3)=1$ and so that $m_1,s$ and $d$ satisfy the conditions $(1)$--$(4)$ of Proposition \ref{lem-r3} with $v=3$.

Let $G_0=\langle \o^{3m_1}, (\o^e\a^s)^2\rangle$ and let $G=N:G_0$, where $N$ is the group of translations of the 1-dimensional
vector space over ${\rm GF}(p^d)$. Let $G(q,m_1,e,s)=G^{(2)}$.
\end{defi}

By a direct computation, we have $(\o^e\a^s)^2=\o^{e(1+p^{d-s})}\a^{2s}$.  Note that (2) of Proposition~\ref{lem-r3}) with $v=3$ implies that $s$ is odd and (4) implies that $2s$ divides $d$. Moreover, $\gcd(m_1, p^s+1)=1$ (by (1) of Proposition~\ref{lem-r3} and $p\equiv -1\ \pmod 3$ (by (2) of Proposition~\ref{lem-r3}). Then $p^{d-s}\equiv -1\ \pmod 3$ and then  $e(1+p^{d-s})(\frac{p^d-1}{p^{2s}-1})\equiv 0\ \pmod {3m_1}$. Hence the standard form of $G_0$ is $\langle  \o^{3m_1}, \o^{e(1+p^{d-s})}\a^{2s}\rangle$.

\begin{lem}\label{family1}
The group $G(p^d,m_1,e,s)$ is primitive and is not the automorphism group of any graph or digraph of order $p^d$.
\end{lem}

\begin{proof}  By Proposition~\ref{lem-r3}, $H$ has exactly two orbits on $V^*$, of lengths $\frac{(p^d-1)}{3}$ and $\frac{2(p^d-1)}{3}$. Note that each orbit of $\langle \o^{3m_1}\rangle$ on $V^*$ has length $\frac{p^d-1}{3m_1}$. From the proof of \cite[Proposition~3.3]{Foulser-Kallaher-solvable-rank-3}, one may see that $\o^e\a^s$ induces the permutation $\eta$ on the set of orbits of $\langle \o^{3m_1}\rangle$ on $V^*$ that is the product of two disjoint cycles of lengths $m_1$ and $2m_1$. Since $m_1$ is odd, $\eta^2$ is a product of three disjoint cycles of equal length $m_1$. It follows that $G_0=\langle \o^{3m_1}, (\o^e\a^s)^2\rangle$ has exactly three orbits on $V^*$, each of which has length $\frac{(p^d-1)}{3}$. By Lemma~\ref{lem-3/2}, $G$ is primitive and hence so is $G(p^d,m_1,e,s)$.

By a direct computation, we have $(\o^e\a^s)^2=\o^{e(1+p^{d-s})}\a^{2s}$, and $(\o^{m_1})^{(\o^e\a^s)^2}=(\o^{m_1})^{\a^{2s}}=\o^{m_1p^{2s}}$.
Since $3\mid (p^{2s}-1)$, we have $\o^{m_1p^{2s}}\langle\o^{3m_1}\rangle=\o^{m_1}\langle\o^{3m_1}\rangle$. This implies that the quotient group $\langle \o^{m_1}, (\o^e\a^s)^2\rangle/\langle\o^{3m_1}\rangle$ is abelian, and hence $G_0\unlhd\langle \o^{m_1}, (\o^e\a^s)^2\rangle$. Clearly, $G_0$ has index $3$ in $\langle \o^{m_1}, (\o^e\a^s)^2\rangle$. Hence either $\langle \o^{m_1}, (\o^e\a^s)^2\rangle$ fixes each $G_0$-orbit on $V^*$ or it cyclically permutes them in a cycle of length 3. Suppose, for a contradiction, that $\langle \o^{m_1}, (\o^e\a^s)^2\rangle$ fixes each $G_0$-orbit on $V^*$.
Then $\langle \o^{m_1}, \o^e\a^s\rangle$ would have the same orbits on $V^*$ as $\langle \o^{3m_1}, \o^e\a^s\rangle$, and hence $\langle \o^{m_1}, \o^e\a^s\rangle$ has exactly two orbits on $V^*$, of lengths $\frac{p^d-1}{3}$ and $\frac{2(p^d-1)}{3}$. Now $\langle \o^{m_1}, \o^e\a^s\rangle$ is also in standard form and so by Proposition \ref{lem-r3} we must have that $\frac{p^d-1}{3}=\frac{a_1(p^d-1)}{m_1}$ with $m_1=va_1$ for some prime $v$. It follows that $v=3$, contradicting $\gcd(3,m_1)=1$. Thus, $\langle \o^{m_1}, (\o^e\a^s)^2\rangle$ is transitive on the set of orbits of $G_0$ on $V^*$.
It follows that all the three non-trivial orbital digraphs of $G$ are isomorphic. In particular, the automorphism group of any of them contains a permutation that interchanges the other two. Hence $G(p^d,m_1,e,s)$ is not the automorphism group of any non-trivial orbital graph or digraph.
\end{proof}

We do not determine the explicit structure of $G(p^d,m_1,e,s)$ but note that all of its orbital digraphs have the same valency. This is not the case for the two examples $\mathcal{G}(m)$ and $\mathcal{H}(m)$ given Section \ref{sec:tensor}. Thus it follows from  Theorem \ref{main-theorem} that for $n\geqslant 2402$ we have that $G(p^d,m_1,e,s)\leqslant \AGammaL_1(p^d)$. This implies that there do exist examples of primitive groups of rank $4$ that are subgroups of $\AGammaL_1(p^d)$ and are not the automorphism group of a graph or digraph. However, we do not attempt a complete classification of such examples as it appears that the required arguments would be very number-theoretic. In the next subsection we provide some further small examples not of the form $G(p^d,m_1,e,s)$.

\subsection{Small degrees}\label{sec:smalleg}

The examples below are numbered according to the database of primitive groups in \textsc{Magma} and do not lie in any of the infinite families already given.

\begin{enumerate}
\item PrimitiveGroup(25,11). In this case $G=5^2:(D_{8}.2)$. Here  all three nontrivial orbital graphs  of $G$ are isomorphic to $H(2,5)$. The point stabiliser in $G$ is not isomorphic to a subgroup of $\GammaL(1,25)$, and the only proper rank 4 subgroup has point stabiliser $Q_8$, which is also not subgroup of $\GammaL(1,25)$. Thus $G$ is not of the form $G(25,m_1,e,s)$ for any $m_1,e$ and $s$.

\item PrimitiveGroup(49,16). Here $G=7^2:\langle \omega^4, \alpha \rangle$ has two nontrivial orbital graphs isomorphic to $H(2,7)$. The other orbital graph has automorphism group $7^2:\langle \omega^2, \alpha\rangle$. Since the orbital graphs are not all isomorphic $G$ is not of the form $G(49,m_1,e,s)$ for any $m_1,e$ and $s$.

\item PrimitiveGroup(64,27).  Here $G=2^6:(3^{1+2}_+:D_8)$ has two orbital graphs of valency 18 and one of valency 27. The first two have automorphism group $2^6:3.A_6.2$, which appears in Table \ref{tab-sub-3}. The orbital graph of valency 27 has automorphism group $2^6.\GO^-_6(2)$, which appears in case (A7) of Table \ref{tab-sub-1}.

\item PrimitiveGroup(81,48).  Here $G=3^4 : \langle \omega^4, \alpha\rangle$ has two orbital graphs of valency 20 and one of valency 40. The first two have automorphism group $3^4 : \mathrm{CO}^-_4(3)$, which appears in case (A7) of Table \ref{tab-sub-1}. The remaining one has automorphism group $3^4 : \langle \omega^2, \alpha \rangle$. Hence the orbital graphs are not all isomorphic and so $G$ is not of the form $G(81,m_1,e,s)$ for any $m_1,e$ and $s$.  

\item PrimitiveGroup(81,77). Here $G=3^4:(2\times Q_8):A_4$ and has two orbital graphs of valency 16 and one of valency 48. The first two are Hamming graphs $H(2,9)$, while the latter has automorphism group $3^4:2^{1+4}.\GO^+_4(2)$ which is a rank 3 group in Table \ref{tab-sub-2}.

\item PrimitiveGroup(81,87). Here $G=3^4: (\GL(1,3) \wr D_8).2$. Here $G$ has one orbital graph of valency 16 and two of valency 32. The first is isomorphic to $H(2,9)$ while the automorphism groups of the other two are isomorphic to $3^4:2^{1+4}.\GO^+_4(2)$, which is a rank 3 group in Table \ref{tab-sub-2}.

\item PrimitiveGroup(121,23). Here $G=11^2:(5\times Q_8) $ with all nontrivial suborbits having length 40. The stabiliser in $G$ of a point is not cyclic and $G$  does not have a proper rank 4 subgroup. Thus $G$ is not of the form $G(121,m_1,e,s)$ for any $m_1, e$ or $s$, as all such groups for $q=121$ are the 2-closure of a primitive group of rank 4 having cyclic point stabilisers. The automorphism group of each of the orbital digraphs is equal to $11^2:(C_{40}:C_2)\leqslant\AGammaL_1(121)$.  

\item PrimitiveGroup(169,41). Here $G=13^2: (3 \times 3:8)$ which has an orbital graph of valency 24 (the Hamming graph $H(2,13)$) and two of valency 72. The latter both have automorphism group $13^2: 12 \circ 2^{1+2}:S_3$, which appears in Table \ref{tab-sub-2}.


\item PrimitiveGroup(625,547). Here $G=5^4:4.A_6$ which has an orbital graph of valency 144 and two of valency 240. The graph of valency 144 has $5^4:4.A_6.2$, which appears in Table \ref{tab-sub-3}, as its automorphism group. The two orbital graphs of valency 240 have $5^4: 4\circ 2^{1+4}\Sp_4(2)$.


\item PrimitiveGroup(2401,663).  Here $G=7^4 : \langle \omega^{10}, \omega^5\alpha \rangle$ and has two suborbits of length 960 and one of length 480.  This is not of type $G(7^4,m_1,e,s)$ for any $m_1,e$ and $s$ as the suborbits are not all of the same length.
The orbital graph of valency 480 has automorphism group $7^4 : \langle \omega^5, \alpha \rangle\leqslant\AGammaL_1(7^4)$. The other two nontrivial orbital graphs have automorphism group $7^4 : (\mathrm Z(\GL_2(49) \circ \SL_2(5))).2$, which appears in Table~\ref{tab-sub-3}.

\item PrimitiveGroup(2401,991). Here $G=7^4:C_6\circ 2^{1+4}\Omega^-(4,2)$ which has two orbital graphs of valency 240 and one of valency 1920.  The first two have automorphism group $7^4:C_6\circ \Sp_4(3)$, which appear in Table \ref{tab-sub-3} as $\mathrm{PSp}_4(3)\cong U_4(2)$, while the last has automorphism group $7^4:C_6\circ 2^{1+4}\GO^-_4(2)$, which appears in Table \ref{tab-sub-2}.

\end{enumerate}

\begin{lem}\label{lem:small}
If $G$ is a $2$-closed primitive group of rank $4$ of degree $n\leqslant 4095$ that is not the automorphism group of a graph or digraph and $G\not\leqslant\AGammaL_1(p^d)$ with $n=p^d$ for some prime $p$, then $G$ is given in Theorem~{\rm\ref{main-theorem}(4)}.  
\end{lem}
\begin{proof}
We work through the \textsc{Magma} database  of primitive groups of degree at most 4095. The ones of rank 4 can then be determined and it can be checked if they are 2-closed and not the automorphism group of a graph or digraph. This results in the groups with label $(n,k)$ being in the set 
$$\{(16, 2),(25, 4), (25, 11), (49, 16), (64, 14), (64, 27), (81, 48), (81, 77), (81, 87), (121, 23), (121, 24),$$
$$(169, 41), (256, 52), (256, 155), (289, 50), (529, 39), (625, 360), (625, 547), (729, 417), (841, 80),$$ $$(1024, 26), (1681, 152), (2209, 51), (2401, 663), (2401, 991), (2809, 81), (3481, 67)\}.$$

The group $(256,155)$ is $\mathcal{H}(2)$ from Lemma \ref{family3} while $(729,417)$ is $\mathcal{G}(3)$ from Lemma \ref{family2}. Of the remaining ones, all but those listed in Theorem \ref{main-theorem} are contained in $\AGammaL(1,p^d)$ and those not listed in this section are  isomorphic to $G(q,m_1,e,s)$ for some $q,m_1,e$ and $s$.  
\end{proof}

\section{A reduction}
In this section, we show that if $G$ is a primitive $2$-closed permutation group of rank $4$ and  degree $n$ that is not the automorphism group of any graph or digraph of order $n$, then $G$ is an affine primitive permutation group.
Almost by definition, we obtain the following observation.

\begin{lem}\label{primitive}
Let $G$ be a primitive $2$-closed permutation group of degree $n$. If $G$ is not the automorphism group of any graph or digraph of order $n$, then $G$ has rank at least $4$. Moreover, if $G$ has rank $4$ and is the automorphism group of a graph or digraph then it is the automorphism group of one of its orbital digraphs.
\end{lem}

\begin{proof} 
If $G$ has rank $2$, then $G\cong S_n$ which is the automorphism group of the complete graph $\K_n$, a contradiction. Assume that $G$ has rank $3$. Then $G$ has two non-trivial orbital digraphs, one of which is the complement of the other. This implies that $G$ is the automorphism group of both of its non-trivial orbital digraphs, a contradiction.  Thus $G$ has rank at least $4$.

Suppose now that $G$ is the automorphism group of some graph or digraph $\Gamma$ and let $E_1$, $E_2$ and $E_3$ be the three nontrivial orbitals of $G$. Then $\Gamma$ is the union of one or two orbital digraphs of $G$. If $\Gamma = (\Omega,E_1 \cup E_2)$, say,  then the complement of $\Gamma$ is the orbital digraph $(\Omega,E_3)$. Since $\Gamma$ and its complement have the same automorphism group it follows that $G$ is the automorphism group of the orbital digraph for $E_3$.
\end{proof}

We will adopt the following assumption for the remainder of this paper.\medskip

\f{\bf Assumption~I.}\
\begin{itemize}
  \item $\Omega$: a non-empty set of size $n\geqslant 4096$,  $\a\in \Omega$,
  \item $G$: a primitive $2$-closed permutation group on $\Omega$ of rank $4$ that is not the automorphism group of any graph or digraph of order $n$,
  \item $B_0=\{\a\}, B_1, B_2, B_3$ are the four orbits of $G_\a$ on  $\Omega$,
  \item $E_i=\{(\a, \b)^g\ |\ g\in G, \b\in B_i\}$ for $i=1,2,3$,
   \item $\Gamma_i=(\Omega, E_i)$ for $i=1,2,3$.
\end{itemize}

The following lemma gives some basic facts for the non-trivial orbital graphs of $G$.

\begin{lem}\label{rank4}
Under Assumption I, the following hold for each $i=1,2,3$:
\begin{enumerate}
  \item [{\rm (1)}]\ $\Aut(\Gamma_i)$ is primitive of rank 3, with $\Aut(\Gamma_i)_\alpha$ suborbits $\{\alpha\}$, $B_i$ and $\cup _{j\neq i} B_j$.
  \item [{\rm (2)}]\  $\Gamma_i$ is a distance-transitive graph of order $n$ and diameter $2$.
\end{enumerate}
\end{lem}
\begin{proof}
Let $i\in \{1,2,3\}$. Since $G$ is primitive and contained in $\Aut(\Gamma_i)$, the group $\Aut(\Gamma_i)$ is primitive. Since $G$ is $2$-closed, and not equal to $\Aut(\Gamma_i)$, we see that the rank of $\Aut(\Gamma_i)$ is less than the rank of $G$. Since $E_i$ is an orbital of $\Aut(\Gamma_i)$ on $\Gamma_i$, the other two orbitals of $\Aut(\Gamma_i)$ must be the trivial orbital and $\cup _{j\neq i} E_j$. Thus  (1) holds.  

Note that $\Gamma_i$ is connected by Lemma~\ref{primitive}. As $\Aut(\Gamma_i)$ has rank three, the vertices of $\Gamma_i$ at distance one from $\alpha$ are those in $B_i$, and the remaining vertices in $\cup_{j\neq i} B_j$ are at distance two, so $\Gamma_i$ has diameter two. It remains to show that $\Gamma_i$ is a graph, i.e. that $B_i$ is self-paired.

Let us assume that $B_i$ is not self-paired, and let $B_j$ be the paired suborbit of $B_i$ and $B_k$ the remaining suborbit of $G$. By (1), $\Aut(\Gamma_i)$ is primitive, and both $B_i$ and $B_j \cup B_k$ are the two non-trivial suborbits of $\Aut(\Gamma_i)$. Since $|B_i|=|B_j|$, we have $|B_i| < |B_j| + |B_k|$. This implies that the two non-trivial suborbits of $\Aut(\Gamma_i)$ are self-paired, and in particular, $B_i$ is self-paired,  a contradiction. Thus, all suborbits of $G$ are self-paired, and each $\Gamma_i$ is a graph.
\end{proof}

To show that $G$ is not almost simple, we shall use a result in \cite{B-G-L-P-S} about the classification of almost simple $\frac{3}{2}$-transitive permutation groups, and a result in \cite{Libeck-Praeger-Saxl2} regarding the classification of those pairs $H, K$ of primitive permutation groups on a set such that $H<K$,  $\soc(H)\neq\soc(K)$, and $H$ and $K$ share a common non-trivial orbital.

\begin{lem}\label{nonsimple}
Under Assumption I, $G$ is not almost simple.
\end{lem}

\begin{proof} Suppose to the contrary that $G$ is almost simple and recall that $n\geqslant 4096$. By Lemma \ref{rank4}, the orbital digraphs $\Gamma_i$ of $G$ are graphs. We shall divide the proof into the following two cases:

\medskip
\f{\bf Case~1}\ There are two orbital graphs,  $\Gamma_1$ and $\Gamma_2$ say, of $G$ such that 
$$ soc(\Aut(\Gamma_1))=\soc(\Aut(\Gamma_2))=\soc(G).$$

Note that $B_1$ and $B_2\cup B_3$ are the two orbits of $\Aut(\Gamma_1)_\a$ on $\Omega\setminus\{\a\}$. Clearly, $\soc(G)_\a\unlhd\Aut(\Gamma_1)_\a$, so all the orbits of $\soc(G)_\a$ on $B_1$ have the same cardinality. Similarly,
all the orbits of $\soc(G)_\a$ on $B_2\cup B_3$ have the same cardinality.
Also, note that $B_2$ and $B_1\cup B_3$ are the two orbits of $\Aut(\Gamma_2)_\a$ on $\Omega\setminus\{\a\}$, and each $B_i$ is an orbit of $G_\a$.
All the orbits of $\soc(G)_\a$ on $B_2$  have the same cardinality. It then follows that all the orbits of $\soc(G)_\a$ on $\Omega\setminus\{\a\}$ have the same cardinality, and hence $\soc(G)$ is $\frac{3}{2}$-transitive. By Proposition~\ref{2/3-trans} or \cite[Theorem~1.2]{B-G-L-P-S},
either $\soc(G)=\Alt(7)$, $n=21$ and $\Omega$ is the set of pairs in $\{1,2,\ldots,7\}$, or $\soc(G)=\PSL_2(q)$, $n=\frac{1}{2}q(q-1)$ and $q=2^f\geq8$. In the former case, $\soc(G)$ has rank $3$, a contradiction. In the latter case, either $G=\soc(G)$ or  $G=\Aut(\PSL_2(2^f))=\PSL_2(2^f).f$, where $f$ is prime. It is shown in \cite[Lemma 6.2]{B-G-L-P-S} that  $\PSL_2(2^f)$ has rank $2^{f-1}$ while $\Aut(\PSL_2(2^f))$ has rank $1+(2^{f-1}-1)/f$. Since $G$ has rank 4 it follows that $G=\PSL_2(8)$ or $\Aut(\PSL_2(2^5))$.  If $G=\PSL_2(8)$ then $\Aut(\Gamma_1)=\PSL_2(8)$ or $\PL_2(8)$. However, the first has rank 4 while the second has rank 2, contradicting the fact that $\Aut(\Gamma_1)$ has rank 3. Similarly, if $G=\Aut(\PSL(2,2^5))$ then $\Aut(\Gamma_1)=G$, contradicting the fact that $G$ is not the automorphism group of a graph.

\medskip
\f{\bf Case~2}\ There is at most one orbital graph of $G$ whose automorphism group has the same socle as $G.$
\medskip

Without loss of generality, we may assume that $\soc(G)\neq\soc(\Aut(\Gamma_1))$ and $\soc(G)\neq\soc(\Aut(\Gamma_2))$.  Since $G$ has rank 4 and $\Aut(\Gamma_1)$ has rank 3 it follows that $B_1$ is the only common orbit of $\Aut(\Gamma_1)_\a$ and $G_\a$ on $\Omega\setminus\{\a\}$. Let $X=\soc(\Aut(\Gamma_1))G$. By Lemma \ref{lem:socles} we have $\soc(X)=\soc(\Aut(\Gamma_1))$ and $X\neq G$. Since $G$ is 2-closed and rank 4 it follows that $X$ has the same orbitals as $\Aut(\Gamma_1)$ and so has rank 3. Since $n\geqslant 4096$, it follows that   $\soc(G)$,  $\soc(\Aut(\Gamma)_1)$ and $|B_1|$ are given by Proposition \ref{prn:4in3}  using $H=X$ and $d=|B_1|$. Similarly, $B_2$ is the only common orbit of $\Aut(\Gamma_2)_\a$ and $G_\a$ on $\Omega\setminus\{\a\}$. Thus $\Aut(\Gamma_1)\neq\Aut(\Gamma_2)$. Moreover, we again see that $\soc(G)$,  $\soc(\Aut(\Gamma_2))$ and $|B_2|$ are also given by Proposition \ref{prn:4in3} using $H=\soc(\Aut(\Gamma_2))G$ and $d=|B_2|$.  In particular, $|B_1|=|B_2|$.

From Proposition \ref{prn:4in3} one may see that $N_{S_n}(\soc(G))$ has rank   $4$.  This implies that $\soc(G)\neq\soc(\Aut(\Gamma_3))$.  Since $G_\a$ and $\Aut(\Gamma_3)_\a$ have $B_3$ as the only orbit on $\Omega\backslash\{\a\}$, we again have that  $\soc(G)$, $\soc(\Aut(\Gamma_3))$ and $|B_3|$ are given by Proposition \ref{prn:4in3}  using $H=\soc(\Aut(\Gamma_3))G$ and $d=|B_3|$. 
It follows that $|B_1|=|B_2|=|B_3|$, and so $n-1=3|B_1|$. However, this is impossible given the possibilities for $n$ and $|B_i|$.\end{proof}

Finally, we shall prove that $G$ is an affine primitive permutation group of rank $4$.

\begin{prop}\label{affine-simple}
Under Assumption I, $G$ is affine.
\end{prop}

\begin{proof} 
Since $G$ is primitive of rank 4, the structure of $G$ is given by cases (1)--(5) of Proposition \ref{prn:rank4}.
By Lemma~\ref{nonsimple}, (2) does not hold. For (3), we use {\sc Magma} to see that $G$ is the automorphism group of one of its orbital graphs, which is against Assumption I. 
Suppose that (4) holds. Then since $G$ has rank 4 and preserves the product structure $\Omega=\Delta^3$, the four orbits of $G_\alpha$ must be
\[
\begin{array}{l}
B_0=\{(\d,\d,\d)\},\\
B_1=\{(\d,\d,\g), (\d,\g, \d), (\g, \d, \d)\ |\ \g\in \Delta\setminus\{\d\}\},\\
B_2=\{(\d,\b,\g), (\b,\d, \g), (\b, \g, \d)\ |\ \b, \g\in \Delta\setminus\{\d\}\},\\
B_3=\{(\mu,\b,\g)\ |\ \mu, \b, \g\in \Delta\setminus\{\d\}\}.
\end{array}
\]
Moreover, $|B_1|=3(n_0-1)$, $|B_2|=3(n_0-1)^2$ and $|B_3|=(n_0-1)^3$. It is easy to see that all orbitals of $G$ are self-paired, and that the orbital graph $\Gamma_1=(\Omega, E_1)$ is of diameter $3$, where $E_1=\{(\a, \b)^g\ |\ g\in G, \b\in B_1\}$. This is impossible by Lemma~\ref{rank4}~(2).

Finally, suppose that (5) holds. Then every orbital graph of $G$ is a Cayley graph \[\Sigma=\Cay(A_5,\{g^{-1}tg\ |\ g\in {A_5}\}),\] where $t=(1, 2)(3, 4), (1,2,3)$ or $(1,2,3,4,5)$.
By Magma~\cite{BCP}, $\Aut(\Sigma)$ has rank $4$. This is again impossible by Lemma~\ref{rank4}~(1).\end{proof}

\begin{lem}\label{lem:istable1}
Under Assumption I with $n\geqslant 4096$,
$\Aut(\Gamma_i)$ is not one of the rank three groups given in Tables  \ref{tab-sub-2} and \ref{tab-sub-3}.
\end{lem}
\begin{proof}
We  construct in \textsc{Magma} each of the rank three groups given in Tables \ref{tab-sub-2} and \ref{tab-sub-3} of degree larger than 4095, and determine if they have any subgroups of rank four.  Only two do and we eliminate them as follows:

\begin{enumerate}
\item $G=3^8: 2^{1+6}.2^4.A_5.2 < 3^8:2^{1+6}\GO_6^-(2)=H$. (In this case $G$ has proper subgroups which are also rank four. However, this implies that they have the same four orbitals as $G$ and hence are not 2-closed.) Here $G_\alpha$ splits the orbit of $H_\alpha$ of length 1440 into one of length $160$ and one of length 1280.  The orbit of length 160 yields the Hamming graph $H(2,3^4)$ whilst $G$ is the automorphism group of the orbital graph of valency 1280.
\item $G=5^6: (4.\GU(3,3).2) <5^6:(4.J_2.2)=H$. Here $G_\alpha$ splits  the orbit of $H_\alpha$ of length 7560 into one of length 1512 and one of length 6048. However, $G$ is the automorphism group of its orbital graph of valency 1512. \qedhere
\end{enumerate}
\end{proof}

The results of this section combined with Lemma \ref{lem:small} allow us to adopt the following assumption, in addition to Assumption I.

\medskip
\f{\bf Assumption~II.}\
\begin{itemize}
  \item $G$ has socle $V$, where $V\cong\mz_p^d$ for some prime $p$ and integer $d$,
  \item $V=\Omega$,
  \item $V^*=V\setminus\{0\}$,
  \item $G_0$ is the stabiliser of the zero vector $0$ in $V$,
  \item $\Aut(\Gamma_i)_0$ is not in Tables \ref{tab-sub-2} and \ref{tab-sub-3} for each $i\in\{1,2,3\}$,
  \item $n=p^d\geqslant 4096$.
\end{itemize}

\section{The orbital graphs}
In this section,  we show that under Assumptions I and II, either $\Gamma_i$ or its complement is isomorphic to a Hamming graph, or $\soc(G)=\soc(\Aut(\Gamma_i))$ for all $i$, and at least one of  the subgroups $\Aut(\Gamma_i)_0$  is contained in $\G L_1(p^d)$.

\begin{lem}\label{not-affine}
Under Assumptions I and II, suppose that $\soc(\Aut(\Gamma_i))=\soc(G)$ for all $i$. Then for some $i\in \{1,2,3\}$ we have that  $\Aut(\Gamma_i)_0$ is of type {\rm(A1)} in Table~{\rm\ref{tab-sub-1}}, that is, $\Aut(\Gamma_i)_0\leqslant\GammaL_1(p^d)$.
\end{lem}

\begin{proof} Suppose to the contrary that  $\Aut(\Gamma_i)_0$ is not of type (A1) in Table \ref{tab-sub-1} for each $i \in \{1,2,3\}$. Let $V=\soc(G)\cong C_p^d$.
Since $\soc(\Aut(\Gamma_i))=V$ for all $i$,  each $\Aut(\Gamma_i)$ is a primitive affine permutation group of rank $3$, and by Theorem \ref{subdegree} and Assumption II we have that  $\Aut(\Gamma_i)_0$ is  listed in Table~\ref{tab-sub-1}.

 Suppose first that $\Aut(\Gamma_i)_0$ is of type (A2) for some $i$. Then $\Aut(\Gamma_i)_0$ preserves the decomposition $V=V_1\oplus V_2$. Thus $\Aut(\Gamma_i)\leqslant \Sym(|V_1|)\wr S_2$ and since $\Aut(\Gamma_i)$ is 2-closed and rank 3 we must have $\Aut(\Gamma_i)=\Sym(|V_1|)\wr S_2$. Since $\Aut(\Gamma_i)$ is affine and $p^d\geqslant 4096$, we have a contradiction. Thus no $\Aut(\Gamma_i)_0$ is of type (A2). It follows from the posssibilities listed in  Table \ref{tab-sub-1}  that $d=2m$ with $m\geqslant 1$. 

Recall that $B_1,B_2$ and $B_3$ are the three orbits of $G_0$ on the set of  non-zero vectors of $V$, and $B_i$ is also an orbit of $\Aut(\Gamma_i)_0$ for $i=1,2,3$. Assume that $|B_1|\leqslant |B_2|\leqslant |B_3|$.
Then $B_1, B_2\cup B_3$ are two orbits of $\Aut(\Gamma_1)_0$, $B_2, B_1\cup B_3$ are two orbits of $\Aut(\Gamma_2)_0$, and $B_3, B_1\cup B_2$ are two orbits of $\Aut(\Gamma_3)_0$. Clearly $|B_1|<|B_2|+|B_3|$ and $|B_2|<|B_1|+|B_3|$.

We split our analysis into the following cases and derive a contradiction in each instance.

\medskip\f{\bf Case~1}\ $p$ divides $|B_i|$ for some $i=1,2,3$. \medskip

Since $|B_1|+|B_2|+|B_3|=p^{2m}-1$, there are at most two $B_i$'s such that $p$ divides $ |B_i|$.

Suppose first that there is only one $B_i$ such that $|B_i|=pk$ for some integer $k$.  Then by Table~1, the other two will have size of the form: $(a+1)(b-1)$, $(c+1)(d-1)$, respectively, for some powers $a,b,c,d$ of $p$. As $|B_1|+|B_2|+|B_3|=p^{2m}-1$, it follows that \[pk+ab-a+b-1+cd-c+d-1=p^{2m}-1,\] contradicting $a,b,c$ and $d$ all being divisible by $p$.

Suppose now that there are two $B_i$'s such that $p$ divides $|B_i|$. Thus at least one of $|B_1|$ or $|B_2|$ is divisible by $p$. If $|B_1|$ is divisible by $p$, then since $|B_1|<|B_2|+|B_3|$, Lemma  \ref{lem:larger} implies that $p=2$, $\Aut(\Gamma_1)_0$  is of type (A6) or (A7), and $|B_1|=(2^m-2^{m-1})(2^m-1)=2^{m-1}(2^m-1)$. Similarly,  if $|B_2|$ is divisible by $p$ then $p=2$ and $|B_2|=(2^m-2^{m-1})(2^m-1)=2^{m-1}(2^m-1)$.

Suppose that $2$ divides both $|B_1|$ and $|B_2|$. Then we have $|B_1|=|B_2|=(2^m-2^{m-1})(2^m-1)$ and so $|B_3|=2^m-1<|B_1|$, a contradiction. Thus we have that $|B_i|$ is coprime to 2 and $|B_j|=(2^m-2^{m-1})(2^m-1)=2^{m-1}(2^m-1)$, where $\{i,j\}=\{1,2\}$.
Then $|B_i|+|B_3|=2^{2m}-1-2^{m-1}(2^m-1)=(2^m+1-2^{m-1})(2^m-1)=(2^{m-1}+1)(2^m-1)$. Hence $|B_i|$ is not divisible by $2^m+1$. Then by checking Table \ref{tab-sub-1}, we see that $|B_i|=(2^a+1)(2^m-1)$ and $|B_3|=2^s(2^t-1)(2^m-1)$, where $a,s,t$ are positive integers satisfying $s+t=m$. So we have \[2^{m-1}=2^a+2^s(2^t-1)=2^a+2^m-2^s.\]
It follows that $2^s-2^a=2^{m-1}$, and hence $a=m-1$ and $s=m$. Thus $t=m-s=0$, contradicting $t>0$.

\medskip\f{\bf Case~2}\ $p\nmid |B_i|$ for every $i=1,2,3$. \medskip

From Table \ref{tab-sub-1}, we can divide this case into the following four subcases:

\medskip\f{\bf Subcase~2.1}\ $|B_1|=(s+1)(p^m-1)$, $|B_2|=(\ell+1)(p^m-1)$ and $|B_3|=(r+1)(p^m-1)$, where $s, \ell, r<p^m$ are powers of $p$.\medskip

As $|B_1|+|B_2|+|B_3|=p^{2m}-1=(p^m-1)(p^m+1)$, it follows that
\[s+\ell+r+2=p^m.\]
Hence $p=2$. 
As $|B_1|\leqslant |B_2|\leqslant |B_3|$, one has $s\leqslant \ell\leqslant r$. Then the above equality implies that $(s,\ell,r,p^m)=(2,2,2,8)$ or $(2,4,8,16)$, which contradicts the assumption that $n=p^{2m}\geqslant 4096$.

\medskip\f{\bf Subcase~2.2}\ $|B_1|=(s-1)(p^m+1)$, $|B_2|=(\ell-1)(p^m+1)$ and $|B_3|=(r-1)(p^m+1)$, where $s, \ell, r<p^m$ are powers of $p$.\medskip

By checking Table~1, we see that in this case, each $\Aut(\Gamma_i)_0$ is of one of types (A6), (A7) or (A11).
Now $|B_1|+|B_2|+|B_3|=p^{2m}-1=(p^m-1)(p^m+1)$, and so 
\[s+\ell+r-2=p^m.\]
Hence $p=2$. Also, as $|B_1|\leqslant |B_2|\leqslant |B_3|$, one has $s\leqslant \ell\leqslant r$. Then the  above equality implies that $s=2$.
Since $\Aut(\Gamma_1)_0$ is of one of types (A6), (A7) or (A11), the subdegrees in  Table \ref{tab-sub-1} imply that $p^d=2^4$, contradicting $n=p^d\geqslant 4096$.

\medskip\f{\bf Subcase~2.3}\ One of $|B_i|$ ($i=1,2,3$) has the form $(s-1)(p^m+1)$, and the other two have the form $(r+1)(p^m-1)$, where $s, r<p^m$ are powers of $p$.\medskip

As $|B_1|+|B_2|+|B_3|=p^{2m}-1=(p^m-1)(p^m+1)$, it follows that $p^m-1$ divides $(s-1)(p^m+1).$
Since $\gcd(p^m-1,p^m+1)$ divides $2$, either $p=2$ and $p^m-1$ divides $s-1$, or $p$ is odd and $\frac{p^m-1}{2}$ divides $s-1$. Since $s<p^m$ it follows that $p$ is odd and $\frac{p^m-1}{2}=s-1$. Thus $p^m-1=2s-2$, contradicting $s$ being a power of $p$.

\medskip\f{\bf Subcase~2.4}\ One of $|B_i| (i=1,2,3)$ has the form $(s+1)(p^m-1)$, and the other two have the form $(r-1)(p^m+1)$, where $s, r<p^m$ are powers of $p$.\medskip

As $|B_1|+|B_2|+|B_3|=p^{2m}-1=(p^m-1)(p^m+1)$, it follows that $p^m+1$ divides $(s+1)(p^m-1).$
Since $\gcd(p^m-1,p^m+1)$ divides $2$,  either $p=2$ and $p^m+1$ divides $s+1$, or $p$ is odd and $\frac{p^m+1}{2}$ divides $s+1$. Since $p^m>s$ it follows that $p$ is odd and $\frac{p^m+1}{2}=s+1$.  Thus $p^m+1=2s+2$, contradicting $s$ being a power of $p$.
\end{proof}

\begin{lem}\label{auto-groups}
Under Assumptions I and II, if there exists an orbital graph, say $\Gamma_1$ such that $\soc(\Aut(\Gamma_1))\neq\soc(G)$, then the following hold:
\begin{enumerate}[$(1)$]
  \item $G_0$ stabilises a pair $\{V_1, V_2\}$ of subspaces of $V$, where $V=V_1\oplus V_2$ and $\dim(V_1)=\dim(V_2)=m$;
  \item  $\Gamma_1$ is isomorphic to the Hamming graph $H(2, |V_1|)$ or its complement with $\Aut(\Gamma_1)\cong (\Sym(V_1)\times\Sym(V_2))\rtimes S_2$;
  \item For $i=2$ and 3, either $\Aut(\Gamma_i)$ is affine with socle $V=\soc(G)$ or $\Aut(\Gamma_i)\cong\Aut(\Gamma_1)$.
\end{enumerate}
\end{lem}

\begin{proof} Note that for each $i$, $G$ and $\Aut(\Gamma_i)$ have a common orbital, that is, the set of  edges of $\Gamma_i$. Since $G$ is not 2-transitive,   \cite[Theorem~1]{Libeck-Praeger-Saxl2} implies that if $\soc(G)\neq\soc(\Aut(\Gamma_i))$ then $G<\Aut(\Gamma_i)\leqslant\Sym(\Omega_0) \wr S_m$ with $\Omega=\Omega_0^m$.  Moreover, as $\Aut(\Gamma_i)$ has rank three it follows that $m=2$ and the three orbitals of $\Aut(\Gamma_i)$ are the three orbitals of the rank three group $\Sym(\Omega_0)\wr S_2$. Thus $\Gamma_i$ is either the Hamming graph $H(2,|\Omega_0|)$ or its complement, and $\Aut(\Gamma_i)=\Sym(\Omega_0)\wr S_2$.  This gives (2) and (3). Now $\AGL_d(p)\cap (\Sym(\Omega_0)\wr S_2)=C_p^d \rtimes (\GL(d/2,p)\wr C_2)$, where $\GL(d/2,p)\wr C_2$ is the stabiliser in $\GL_d(p)$ of the decomposition $V=V_1\oplus V_2$, where  $\dim(V_1)=\dim(V_2)=m=d/2$.  Thus $G_0$ stabilises this decomposition and we have part (1).
\end{proof}

\section{$G_0$ is imprimitive}

We begin with the following two lemmas

\begin{lem}\label{B1-orbit}
Under Assumptions I and II, suppose that $G_0$ stabilises a pair $\{V_1, V_2\}$ of subspaces of $V$, where $V=V_1\oplus V_2$ and $\dim(V_1)=\dim(V_2)=m$. Then $(V_1\cup V_2)\setminus\{0\}$ is an orbit of $G_0$ on the non-zero vectors of $V$.
\end{lem}

\begin{proof} Since $G$ is primitive of rank $4$, $G_0$ acts irreducibly on $V$, and so $G_0$ acts transitively on $\{V_1, V_2\}$.
Let $(G_0)_{V_i}$ be the setwise stabiliser of $V_i$ in $G_0$ for $i=1,2$.

Suppose, for a contradiction, that $(G_0)_{V_1}$ is intransitive on $V_1\setminus\{0\}$. Then $(G_0)_{V_2}$ is also intransitive on $V_2\setminus\{0\}$.
Thus $G_0$ has at least 2 orbits on $(V_1\cup V_2)\setminus\{0\}$.
Since $G_0$ has exactly 3 orbits on $V\setminus\{0\}$, we conclude that
$G_0$ has exactly 2 orbits on $(V_1\cup V_2)\setminus\{0\}$ and is transitive on $V\setminus(V_1\cup V_2)$.
Let $K$ be the kernel of $G_0$ acting on $\{V_1,V_2\}$. Then $K=(G_0)_{V_1}$ and  $K\lhd G_0$ with $G_0/K\cong \mz_2$. Moreover, $K$ is intransitive on both $V_1$ and $V_2$. Let $u,v\in V_1$
be in different orbits of $K$ and let $a,b\in V_2$ be in different orbits of $K$.
Since $K$ fixes $V_1$ and $V_2$, the four vectors
\[u+a,\ u+b,\ v+a,\ v+b\]
are pairwise inequivalent under the action of $K$.
This implies that $K$ has at least 4 orbits on $V\setminus(V_1\cup V_2)$  and so $G_0$ has at least two orbits on $V\setminus(V_1\cup V_2)$, which is a contradiction.

Thus $(G_0)_{V_1}$ is transitive on $V_1\setminus\{0\}$ and hence $(G_0)_{V_2}$ is also transitive on $V_2\setminus\{0\}$.
It follows that $(V_1\cup V_2)\setminus\{0\}$ is an orbit of $G_0$. \end{proof}

\begin{lem}\label{B1-orbit-new}
Let $H\leq\GL_{2m}(p)$ be such that $H$ stabilises the pair $\{V_1, V_2\}$  of subspaces of $V$, where $V=V_1\oplus V_2$ and $\dim(V_1)=\dim(V_2)=m$.
Suppose that $(V_1\cup V_2)\setminus\{0\}$ is an orbit of $H$ on the non-zero vectors of $V$. Then every orbit of $H$ has length divisible by $|V_1|-1$.
\end{lem}

\begin{proof} Let $K$ be the kernel of $H$ acting on $\{V_1,V_2\}$ and note that $K$ is transitive on $V_1 \setminus \{0\}$ and $V_2 \setminus \{0\}$. Take a non-zero vector $v\in V_1$. It is clear that every orbit of $K$ on $V\setminus(V_1\cup V_2)$ has the form $\{(v+x)^k\ |\ k\in K\}$ for some $x\in V_2\backslash\{0\}$.  Then 
\[|\{(v+x)^k\ |\ k\in K\}|=|K: K_{v}\cap K_x|=|K: K_{v}||K_v:K_{v}\cap K_x|=(|V_1|-1)|K_v:K_{v}\cap K_x|.\]
This implies that every orbit of $K$  on $V\setminus (V_1\cup V_2)$ has length divisible by $|V_1|-1$. As $|H: K|=2$, every orbit of $H$ on $V\setminus (V_1\cup V_2)$ is either  an orbit of $K$ or is a union of two orbits of $K$ of equal length. Thus the result follows.
\end{proof}

For the remainder of this section, we shall make the following assumption:\medskip

\f{\bf Assumption~III.}\
\begin{itemize}
 \item Suppose that $G_0$ stabilises a pair $\{V_1, V_2\}$ of subspaces of $V$, where $V=V_1\oplus V_2$ and $\dim(V_1)=\dim(V_2)=m$.
 \item $B_1=(V_1\cup V_2)\setminus\{0\}$.
 \item $\Aut(\Gamma_i)_0$ is not a subgroup of $\GammaL_1(p^d)$ for any $i$.
\end{itemize}

We shall first consider the case where either $\Gamma_2$ or $\Gamma_3$ is isomorphic to $\Gamma_1$.

\begin{lem}\label{B1-B2}
Under Assumptions I,  II and III, if either $\Gamma_2$ or $\Gamma_3$ is isomorphic to $\Gamma_1$, then $G\cong\GG(m)$ as given in Lemma \ref{family2}. 
\end{lem}

\begin{proof} Recall the three orbits of $G_0$ are $B_1, B_2$ and $B_3$ with  $B_1=(V_1\cup V_2)\setminus\{0\}$ and   $|B_2|+|B_3|=(p^m-1)^2$. Assume that $|B_2|\leqslant |B_3|$. By Lemma~\ref{B1-orbit-new}, we have $p^m-1$ divides $\gcd(|B_2|, |B_3|)$.

If $\Gamma_2\cong\Gamma_3\cong\Gamma_1$, then $|B_2|=|B_3|=2(p^m-1)$ and so $\frac{(p^m-1)^2}{2}=2(p^m-1)$.
It follows that $p^m-1=4$, contrary to the  assumption that $p^d=p^{2m}\geqslant 4096$.
Consequently, we have $\Gamma_2\ncong\Gamma_3$. If $\Gamma_3\cong\Gamma_1$, then we must have $|B_3|=2(p^m-1)$ and then $|B_2|=(p^m-3)(p^m-1)$. Since $|B_2|\leq|B_3|$, it follows that $p^m-3\leqslant 2$ and hence $p^m=5$ or $4$. Again, this is contrary to our assumption $p^d=p^{2m}\geqslant 4096$. Thus $\Gamma_2\cong\Gamma_1$ and so $|B_3|=(p^m-3)(p^m-1)$. Since $n\geqslant 4096$ we may assume that $p^m-3\geqslant 4$. Note that $|B_1|$ and $|B_2|$ are the valency of the Hamming graph $H(2,|V_1|)$ and so by comparing valencies, $\Gamma_3$ is not the complement of $H(2,|V_1|)$. Thus by  Lemma~\ref{auto-groups}, $\Aut(\Gamma_3)$ is affine. Hence  $\Aut(\Gamma_3)$ is one of the groups in Proposition~\ref{subdegree}. Clearly, $|B_3|=(p^m-3)(p^m-1)> 4(p^m-1)=|B_1\cup B_2|$ and $B_1\cup B_2$ is an orbit of $\Aut(\Gamma_3)_0$.

By Assumption III, we have that $\Aut(\Gamma_3)_0\not\leqslant\GammaL_1(p^{2m})$.  As $p^{2m}\geq4096$, we may assume that $(p,m)\neq (3,2), (2,3)$. Since $|B_1\cup B_2|=4(p^m-1)$ and $|B_3|=(p^m-3)(p^m-1)$,  we have only one possibility for the cases (A3)--(A11) in Table \ref{tab-sub-1}, namely that
 $p=3$, $m\geqslant 2$ and $\Aut(\Gamma_3)_0$ is of type (A3)--(A5) in Table \ref{tab-sub-1}.
Thus  $V=X\otimes Y$ with dim$(X)=2$ and dim$(Y)=m$, and  $\Aut(\Gamma_3)_0\leqslant \GL_2(3)\circ\GL_m(3)=:M$.
Moreover, by \cite[Lemma~1.1]{Liebeck-affine},  the orbits of $M$ are 
$$\{x\otimes y\mid x\in X, y\in Y\} \quad \textrm{and} \quad \{x_1\otimes y_1+x_2\otimes y_2\mid X=\langle x_1,x_2\rangle \textrm{ and } \dim(\langle y_1,y_2\rangle)=2 \},$$
and since $\Aut(\Gamma_3)$ has rank 3, these must also be the orbits of $\Aut(\Gamma_3)_0$.
Since $|B_1\cup B_2|=4(3^m-1)$, we have
$$B_1\cup B_2= \{x\otimes y\mid x\in X, y\in Y\}.$$ 
Now $V=V_1\oplus V_2$ with $B_1=(V_1\cup V_2)\backslash\{0\}$.  Suppose that $x\otimes y,x'\otimes y'\in V_1$. Since $\dim(V_1)=m\geqslant 3$ we may choose $y$ and $y'$ such that $\dim\langle y,y'\rangle=2$. Now $x\otimes y+x'\otimes y'\in V_1$. Then as $V_1$ only contains simple tensors it follows that $\dim\langle x,x'\rangle=1$. Hence $V_1=\langle x_1\rangle \otimes Y$ and $V_2=\langle x_2\rangle\otimes Y$ where $X=\langle x_1,x_2\rangle$. Since $G$ is 2-closed it follows that $G_0$ is the subgroup of $M$ fixing $\{V_1,V_2\}$ and so $G_0\cong D_8\circ \GL_m(3)$. Thus $G\cong\GG(m)$.
 \end{proof}

Next we show that in the remaining cases we must have that  both $\Aut(\Gamma_2)_0$ and $\Aut(\Gamma_3)_0$ are of type (A3)--(A5)

\begin{prop}\label{affine-impri}
Under Assumptions I, II and III, if neither $\Gamma_2$ nor $\Gamma_3$ is isomorphic to $\Gamma_1$ then both $\Aut(\Gamma_2)_0$ and $\Aut(\Gamma_3)_0$ are of type {\rm(A3)--(A5)} of Proposition~{\rm\ref{subdegree}}.
\end{prop}
\begin{proof} By Lemma \ref{B1-orbit} we may assume that  $B_1=(V_1\cup V_2)\setminus\{0\}$.
Then by Lemma \ref{lem:hamming}, the orbital graph $\Gamma_1$ is isomorphic to the Hamming graph $H(2, |V_1|)$, and $\Aut(\Gamma_1)\cong (\Sym(V_1)\times\Sym(V_2))\rtimes S_2$.
 Then $|B_2|+|B_3|=|V|-1-2(p^m-1)=(p^{m}-1)^2$. 
 
 Suppose that neither $\Gamma_2$ nor $\Gamma_3$ are isomorphic to $\Gamma_1$.  Since they have valency less than $|V|-1-|B_1|$ they are not isomorphic to the complement of $\Gamma_1$ either. Thus
  Lemma~\ref{auto-groups} implies that $\Aut(\Gamma_i)$ is affine for $i=2$ and 3. Now $\Aut(\Gamma_2)$ and $\Aut(\Gamma_3)$ are point stabilisers of rank 3 affine groups with suborbits $B_2$ and $B_3$ respectively. By Assumption III, neither $\Aut(\Gamma_2)_0$ nor $\Aut(\Gamma_3)_0$ are contained in $\GammaL_1(p^{2m})$. Moreover, neither are of type (A2) (as otherwise one of $\Gamma_2$ or $\Gamma_3$ would be isomorphic to $\Gamma_1$ or its complement.). Since  $|B_2|+|B_3|<p^{2m}-1$, it follows from  Proposition~\ref{subdegree},  that $|B_2|$ and $|B_3|$ appear in different rows of Table \ref{tab-sub-1}, but not from the rows corresponding to types (A1) or (A2).  The subdegrees listed there are of one of the following forms: $$ (p^m-1)(s-1),(p^m+1)(s-1),(p^m-1)(p^m-s),(p^m+1)(p^m-s),$$
where $1<s<p^m$ is a power of $p$. By Lemma \ref{B1-orbit-new}, $p^m-1$ divides $\gcd(|B_2|, |B_3|)$. Moreover, $\gcd(p^m-1,p^m+1)$ divides $2$, with equality if and only if $p$ is odd. 

Suppose first that either $|B_2|$ or $|B_3|$ is of the form $(p^m+1)(s-1)$. Since $p^m-1$ divides $\gcd(|B_2|, |B_3|)$ and  noting that $1<s<p^m$, we deduce that $p$ is odd and $\frac{p^m-1}{2}=s-1$. Thus $p^m-1=2s-2$ and so $p^m-2s=-1$, contradicting the fact that $s>1$ is a power of $p$. 

Next suppose that either $|B_2|$ or $|B_3|$ is of the form $(p^m+1)(p^m-s)$. Then we deduce that $p$ is odd and $\frac{p^m-1}{2}=p^m-s$. Thus $p^m-1=2(p^m-s)$, which contradicts the fact that $s>1$ is a power of $p$.

Thus it remains to consider the following possibilities.
\medskip

\f{\bf Case~1}\ $|B_2|=(s+1)(p^m-1)$ and $|B_3|=(r+1)(p^m-1)$, where $s$ and $r$ are powers of $p$, and $s,r< p^m$.\medskip

In this case, we have $(p^m-1)^2=|B_2|+|B_3|=(s+r+2)(p^m-1)$, and hence $p^m-1=s+r+2$. It follows that $p^m=s+r+3$. As $s$ and $r$ are two prime power of $p$ and $s,r< p^m$, one has $s=r=3$, and hence $p^m=9$. Consequently, we have $p^{2m}=81$, contrary to our assumption that $n\geqslant 4096$.

\medskip
\f{\bf Case~2}\ $|B_2|=(p^m-1)(p^m-s)$ and $|B_3|=(p^m-1)(p^m-r)$ where  $s$ and $r$ are powers of $p$, and $s, r^t<p^m$.

In this case we have $(p^m-1)^2=|B_2|+|B_3|=(p^m-1)(2p^m-s-r)$. Thus $p^m-1=2p^m-s-r$, contradicting $s$ and $r$ being powers of $p$ with $r,s>1$. 

\medskip
\f{\bf Case~3}\ $|B_2|=(s+1)(p^m-1)$ and $|B_3|=(p^m-r)(p^m-1)$, where $s$ and $r$ are powers of $p$, and $s, r<p^m$.\medskip

In this case, we have $s+1+p^m-r=p^m-1$, and then $s+2=r$. Since $s=2$ and $n=p^{2m}\geqslant 4096$ it follows from the list of subdegrees in Table \ref{tab-sub-1} that  $\Aut(\Gamma_2)_0$ is of type (A3)--(A5).  Since $r=4$ we see that $\Aut(\Gamma_3)_0$ is not of types (A9) or (A10). Moreover,  if  $\Aut(\Gamma_3)_0$ is of types (A6)--(A8), then in all these cases, the fact that $r=4$ implies that $n<4096$, a contradiction. Hence $\Aut(\Gamma_3)_0$ is also of type (A3)--(A5). This completes the proof.
\end{proof}

Finally, we show that $G\cong\mathcal{H}(m/2)$.
\begin{lem}\label{B2-B3}
Suppose that Assumptions I, II and III hold.  If $\Aut(\Gamma_2)_0$ and $\Aut(\Gamma_3)_0$ are both of type (A3)--(A5) of Proposition~\ref{subdegree}, then $G\cong\mathcal{H}(m/2)$.
\end{lem}

\begin{proof}
Without loss of generality, suppose that $|B_2|\leqslant |B_3|$. Since $\Aut(\Gamma_2)_0$ is of  type (A3)--(A5) of Proposition~\ref{subdegree} we have that $p^d=p^{2m}=q^{2\ell}$ for some power $q$ of $p$ and positive integer $\ell>1$. Moreover, $\Aut(\Gamma_2)_0$ preserves a decomposition $V=U\otimes W$ where $U$ and $W$ are $\GF(q)$-spaces  with dimension 2 and $\ell$ respectively. By Table~1, the sizes of the two orbits of $\Aut(\Gamma_2)_0$ on $V^*$ are:
\[(q+1)(q^\ell-1) \quad \textrm{and} \quad q(q^{\ell-1}-1)(q^\ell-1)\]
and by \cite[Lemma~1.1]{Liebeck-affine}, the corresponding $\Aut(\Gamma_2)_0$ orbits are 
\[ \{u\otimes w\mid u\in U,w\in W\}\backslash\{0\} \quad \textrm{and} \quad \{u_1\otimes v_1+u_2\otimes v_2\mid U=\langle u_1,u_2\rangle \textrm{ and } \mathrm{dim}(\langle v_1,v_2\rangle)=2\}.\]
As $p^d\geq4096$, we have $(p, m)\neq (2,2)$, and hence $(q+1)(q^\ell-1)\leqslant q(q^{\ell-1}-1)(q^\ell-1)$. Since $|B_2|\leqslant |B_3|$ it follows that $|B_2|=(q+1)(q^\ell-1)$. Hence $B_2=\{u\otimes w\mid u\in U,w\in W\} $.

Next, since  $\Aut(\Gamma_3)_0$ is also of type (A3)--(A5) then $|B_3|=(r+1)(r^s-1)$ or $r(r^{s-1}-1)(r^s-1)$, where $r^s=p^m$. Note that $|B_2|+|B_3|=(p^m-1)^2$. Suppose first that $|B_3|=(r+1)(r^s-1)$. Then
$(p^m-1)^2=|B_2|+|B_3|=(q+r+2)(p^m-1)$, and then $q+r+3=p^m$. It follows that $q=r=3$ and $(p,m)=(3,2)$, which is impossible because $p^{2m}\geq4096$. 
Thus $|B_3|=(r^s-r)(r^s-1)$. Then
$(p^m-1)^2=|B_2|+|B_3|=(q+p^m-r+1)(p^m-1)$, and then $q+2=r$. It follows that $(q, r)=(2,4)$. Thus $4^{s}=2^{m}=2^\ell$ and so $m=\ell=2s$. Then $|B_2|=3(2^m-1)$ and $|B_3|=(2^m-4)(2^m-1)$. Since $p^{2m}\geqslant 4096$, one has $m=2s\geqslant 6$.

Now $\Aut(\Gamma_3)_0$ stabilises a decomposition $V=X\otimes Y$, where $X$ and $Y$ are $2$- and $s$-dimensional $\GF(4)$-subspaces of $V$, respectively.
By \cite[Lemma~1.1]{Liebeck-affine},  the orbits of $\Aut(\Gamma_3)_0$ are 
$$\{x\otimes y\mid x\in X, y\in Y\}\backslash\{0\} \quad \textrm{and} \quad \{x_1\otimes y_1+x_2\otimes y_2\mid X=\langle x_1,x_2\rangle \textrm{ and } \dim(\langle y_1,y_2\rangle)=2 \}.$$
Since $B_1\cup B_2$ is an orbit of $\Aut(\Gamma_3)_0$ of size $5(2^m-1)$, we have
$$B_1\cup B_2= \{x\otimes y\mid x\in X, y\in Y\}.$$ 
Moreover, $\Aut(\Gamma_3)_0$ is the stabiliser of the decomposition $X\otimes Y$ in $\GammaL_{m}(4)$ and so $\Aut(\Gamma_3)_0=(\GL_{2}(4)\circ\GL(m/2,4)).2$

Recall that $V=V_1\oplus V_2$ with $B_1=(V_1\cup V_2)\backslash\{0\}$ where $V_1$ and $V_2$ are $m$-dimensional $\GF(2)$-subspaces of $V$.  Take $x\otimes y, x'\otimes y'\in V_1$. Since $\dim(V_1)=m\geqslant 6$, we can choose $y$ and $y'$ such that $\langle y, y'\rangle$ is a $2$-dimensional $\GF(4)$-subspace. Now $x\otimes y+x'\otimes y'\in V_1$. Then as $V_1$ only contains simple tensors, it follows that $\langle x, x'\rangle$ is a $1$-dimensional $\GF(4)$-subspace. Hence, we have $V_1=\langle x_1\rangle\otimes Y$ for some $x_1\in X$ and similarly, 
$V_2=\langle x_2\rangle\otimes Y$ where $X=\langle x_1,x_2\rangle$. 
Then \[B_2=(\langle x_1+x_2\rangle\otimes Y\cup\langle x_1+\lambda x_2\rangle\otimes Y\cup\langle x_1+\lambda^2x_2\rangle\otimes Y)\backslash\{0\}\]
where $\lambda$ is a generator of the multiplicative group of $\GF(4)$. Hence $G_0$ is the stabiliser in $\Aut(\Gamma_3)_0$ of the partition $V=V_1\oplus V_2$ and so $G_0=((C_3\wr C_2)\circ\GL(m/2,4)).2$. Thus $G\cong\mathcal{H}(m/2)$ as in Lemma \ref{family3}.
\end{proof}

\section{Proof of Theorem~\ref{main-theorem}}

By Lemmas~\ref{family2}, \ref{family1} and \ref{lem:small}, the groups $\GG(m)$, $\mathcal{H}(m)$ and the seven small groups in Theorem~\ref{main-theorem}~(3) are 2-closed groups of rank 4 that are not the automorphism group of a digraph. Conversely, suppose that $G$ is a 2-closed group of rank 4 of degree $n$ that is not the automorphism group of a digraph such that $G\not\leqslant \AGammaL_1(p^d)$ with $n=p^d$. If $n<4096$ then Lemma \ref{lem:small} implies that $G$ is given by the theorem. Suppose that $n>4096$.  By Lemma \ref{primitive} we may assume that Assumption I holds. Proposition~\ref{affine-simple} then implies that $G$ is affine with socle $V\cong\mz_p^d$, and Lemma~\ref{lem:istable1} then allows us to assume that Assumption II holds. Since $G$ is not a subgroup of $\AGammaL_1(p^d)$ we have that $\Aut(\Gamma_i)\not\leqslant\AGammaL_1(p^d)$ for any $i$ and so by Lemma \ref{not-affine}, $\soc(G)\neq\soc(\Aut(\Gamma_i))$ for some $i$. Then Lemma \ref{auto-groups} implies that Assumption III holds.   If either $\Gamma_2$ or $\Gamma_3$ is isomorphic to $\Gamma_1$, then Lemma \ref{B1-B2} implies that $G=\mathcal{G}(m)$ where $p^d=p^{2m}$. On the other hand, if neither $\Gamma_2$ nor $\Gamma_3$ are isomorphic to $\Gamma_1$, then Lemmas \ref{affine-impri} and \ref{B2-B3} imply that $G\cong \mathcal{H}(m/2)$ where $p^d=p^{2m}$ and $m$ is even.

\subsection*{Acknowledgements} This work was supported by the National Natural Science Foundation of China (11671030, 12071023, 12161141005), the 111 Project of China (B16002) and the Fundamental Research Funds for the Central Universities (2022JBCG003). The first and second authors were supported by Australian Research Council Discovery Project Grant DP190101024. They thank Beijing Jiaotong University for the generous hospitality during their respective visits, and thank the Isaac Newton Institute for Mathematical Sciences, Cambridge, for support and hospitality during the programme Groups, Representations and Applications: New Perspectives (supported by EPSRC grant no. EP/R014604/1), where further work on this paper was undertaken.

{}

\end{document}